\ifdefined\SIADSREVIEW
  \documentclass[10pt]{article}
\else
  \documentclass[11pt]{article}
\fi

\newif\ifsiadsreview
\ifdefined\SIADSREVIEW
  \siadsreviewtrue
\else
  \siadsreviewfalse
\fi

\usepackage[T1]{fontenc}
\usepackage[utf8]{inputenc}
\usepackage{lmodern}
\usepackage{geometry}
\usepackage{microtype}
\usepackage{amsmath,amssymb,amsthm,mathtools,mathrsfs}
\usepackage{bm}
\usepackage{booktabs}
\usepackage{graphicx}
\usepackage{enumitem}
\usepackage{float}
\usepackage{placeins}
\ifsiadsreview
  \usepackage[mathlines]{lineno}
\fi
\usepackage{tikz}
\usepackage[
  colorlinks=true,
  allcolors=blue,
  pdftitle={Sharp Cyclicity of Neutral Entry--Exit Cycles with Quadratic Grazing},
  pdfauthor={Haibo Lu},
  pdfsubject={Piecewise-smooth slow--fast dynamics and local cyclicity},
  pdfkeywords={slow-fast system, entry-exit relation, piecewise-smooth dynamical system, grazing bifurcation, cyclicity, relaxation oscillation}
]{hyperref}

\usetikzlibrary{arrows.meta,calc,decorations.markings,positioning,shapes.geometric}

\ifsiadsreview
\else
\fi
\setlist{topsep=4pt,itemsep=2pt,parsep=2pt}
\numberwithin{equation}{section}

\newtheorem{theorem}{Theorem}[section]
\newtheorem{proposition}[theorem]{Proposition}
\newtheorem{lemma}[theorem]{Lemma}
\newtheorem{corollary}[theorem]{Corollary}
\theoremstyle{definition}
\newtheorem{definition}[theorem]{Definition}
\newtheorem{hypothesis}[theorem]{Hypothesis}
\theoremstyle{remark}
\newtheorem{remark}[theorem]{Remark}

\newcommand{\R}{\mathbb R}
\newcommand{\eps}{\varepsilon}
\newcommand{\dd}{\,\mathrm d}

\newcommand{\sgn}{\operatorname{sgn}}
\newcommand{\Cyc}{\operatorname{Cyc}}
\newcommand{\siadsneedspace}[1]{%
  \ifsiadsreview
    \par\begingroup
    \dimen0=\pagegoal
    \advance\dimen0 by -\pagetotal
    \ifdim\dimen0<#1\relax
      \endgroup\newpage
    \else
      \endgroup
    \fi
  \fi
}

\title{\Large Sharp Cyclicity of Neutral Entry--Exit Cycles with Quadratic Grazing}
\author{Haibo Lu\thanks{Shanghai Institute of Technology.  Email:
\href{mailto:luhaibo1985@gmail.com}{luhaibo1985@gmail.com}.  ORCID:
\href{https://orcid.org/0009-0000-2717-5968}{0009-0000-2717-5968}.}}
\date{July 2026}

\begin{document}

\ifsiadsreview
  \linenumbers
\fi
\maketitle

\begin{abstract}
We determine how many limit cycles can bifurcate when an invariant-line
entry--exit delay and a quadratic grazing of the switching line occur in the
same planar slow--fast return circuit. We consider continuous piecewise-smooth
systems in which these mechanisms are separated by regular flow segments. For
positive \(\eps\), the physical local cyclicity is at most two when the
smooth-reference curvature
and leading grazing-mismatch coefficient have the same sign, and at most three
when their signs are opposite. Under a rank-two unfolding, the applicable
bound is attained in each sign class for every sufficiently small fixed
\(\eps>0\). To prove these statements, we derive, to every
prescribed finite order near a balanced neutral cycle, a parameter-uniform
Poincar\'e displacement expansion and the one-sided derivative estimates
required for zero counting. Its leading nonsmooth contribution is the
classical three-halves grazing term, whose coefficient is given explicitly by
contact, branch-mismatch, and global-transmission data. A polynomial family
realizes both sign cases. We also construct a qualitatively admissible quartic
extension of the cutoff Gause response class introduced by
Kooij--Zegeling.  Validated
Poincar\'e maps and interval Newton certify two explicit nonzero singular
balanced-neutral parameter points at \(\varrho=\pm1/20\), one in each sign
class.  These singular certificates are distinct from the ordinary
positive-\(\eps\) orbit illustrations.  Our
contribution is the sign-sharp, parameter-uniform composition of the classical
grazing law with an entry--exit passage at unit multiplier.
\end{abstract}



\section{Introduction}

Entry--exit delay and grazing are individually familiar singular mechanisms.
Their interaction at a neutral return is not.  The difficulty is already
visible in the scalar displacement family
\[
 F_{a,b}(q)=a+bq+cq^2+dq_+^{3/2},
\tag{1.1}
\]
with \(q_+=\max\{q,0\}\).  Here \(a\) and \(b\) are the two unfolding
coordinates, while \(c\ne0\) and \(d\ne0\) are fixed coefficients of the
base circuit.  A zero of \(F_{a,b}\) is a fixed point of the selected local
first-return map and therefore represents a local limit cycle.  On the
negative side the displacement is ordinarily smooth.  On the positive side
it is naturally smooth only after the ramification \(q=s^2\).  At \(a=b=0\),
\[
 F_{0,0}''(q)=2c\quad(q<0),
 \qquad
 \left.\frac{\dd}{\dd s}F_{0,0}'(s^2)\right|_{s=0^+}
 =\frac32d.
\tag{1.2}
\]
Thus the two sides have different fold scales.  If \(c\) and \(d\) have the
same sign, their one-sided monotonicities are compatible; if their signs are
opposite, the two fold mechanisms coexist and one additional zero can occur.

\begin{quote}
\textbf{Informal main theorem.}
Suppose a separated entry--exit/grazing return circuit closes on its reference
orbit with multiplier one.  Let \(c\ne0\) be its smooth quadratic return
curvature and \(d\ne0\) the coefficient of the leading continuous-grazing
correction in \((1.1)\).  Its local cyclicity is at most two when \(cd>0\) and
at most three when \(cd<0\).  Under the rank-two unfolding condition, the
corresponding upper bound is attained by physical limit cycles for every
sufficiently small positive slow parameter.
\end{quote}

Here \emph{physical} means \(\eps>0\): the counted zeros correspond to
genuine limit cycles whose orbits follow the actual piecewise field through
the grazing box, rather than the smooth reference continuation used to
isolate the grazing correction.  The mechanisms are \emph{separated} in
the geometric sense: the entry--exit and grazing passages lie in disjoint
local boxes and are connected by regular flow tubes.  This separation still
allows the two singular effects to interact in the same first-return map.

The purpose of the paper is to derive (1.1), with the parameter and derivative
control needed for this zero count, from the genuine slow--fast geometry in
Figure~\ref{fig:geometry}.

The figure should be read by mathematical role.  Panel~(a) supplies the
smooth entry--exit passage that enters the reference return; its composition
with the regular tubes contributes to the smooth curvature \(c\).
Panel~(b) compares the actual crossing orbit with the minus-field reference
continuation and produces the one-sided correction \(dq_+^{3/2}\).  The
coefficient \(d\) also contains the regular transmission from the grazing box
back to the return section.

\ifsiadsreview
  \begin{figure}[p]
\else
  \begin{figure}[!b]
\fi
\centering
\begin{tikzpicture}[>=Latex,font=\small,
 section/.style={very thick,gray!65},
 singular/.style={gray!70,densely dashed,thick},
 physical/.style={violet!85!black,very thick},
 normalarrow/.style={->,thin,gray!85!black}]

 \node[anchor=west] at (-6.55,7.05)
   {(a) entry--exit passage in \((u,v)\)};
 \draw[blue!65!black,very thick] (-3.90,6.62)--(-3.90,5.34);
 \draw[red!70!black,very thick,dashed] (-3.90,5.16)--(-3.90,3.88);
 \node[anchor=east,blue!65!black] at (-4.16,5.93)
   {attracting \(S_0^a\)};
 \node[anchor=east,red!70!black] at (-4.16,4.57)
   {repelling \(S_0^r\)};
 \fill (-3.90,5.25) circle (1.25pt);
 \node[anchor=east] at (-4.16,5.25)
   {stability exchange \(v=0\)};

 \draw[blue!65!black,thick,->] (-3.90,6.16)--(-3.90,5.66);
 \draw[red!70!black,thick,->] (-3.90,4.84)--(-3.90,4.34);

 \draw[section,densely dashed] (5.20,3.76)--(5.20,6.72);
 \draw[section] (5.20,6.10)--(5.20,6.64);
 \draw[section] (5.20,3.86)--(5.20,4.40);
 \node[above=2pt] at (5.20,6.72) {\(u=u_1\)};
 \node[right=4pt] at (5.20,6.58) {\(v\in V_0\)};
 \node[right=4pt] at (5.20,3.90) {\(v\in V_1\)};

 \fill (-3.90,6.34) circle (1.15pt)
   node[above left=2pt] {\(\widetilde v_0^*\)};
 \fill (-3.90,4.16) circle (1.15pt)
   node[below left=2pt] {\(\widetilde v_1^*\)};
 \fill (5.20,6.24) circle (1.15pt)
   node[right=3pt,yshift=-8pt] {\(v_0^*\)};
 \fill (5.20,4.24) circle (1.15pt)
   node[right=3pt,yshift=8pt] {\(v_1^*\)};
 \draw[singular,postaction={decorate},
   decoration={markings,mark=at position .54 with {\arrow{>}}}]
   (5.20,6.24)
   .. controls (2.50,6.22) and (-1.50,6.29) .. (-3.90,6.34);
 \draw[singular,postaction={decorate},
   decoration={markings,mark=at position .54 with {\arrow{>}}}]
   (-3.90,4.16)
   .. controls (-1.45,4.10) and (2.55,4.22) .. (5.20,4.24);

 \draw[normalarrow] (-1.30,5.82)--(-3.22,5.82);
 \node[anchor=west] at (-1.15,5.82) {normal attraction};
 \draw[normalarrow] (-3.22,4.68)--(-1.30,4.68);
 \node[anchor=west] at (-1.15,4.68) {normal repulsion};

 \draw[physical]
   (5.20,6.48)
   .. controls (2.70,6.48) and (-2.75,6.45) .. (-3.55,6.10)
   .. controls (-3.68,5.77) and (-3.68,4.58) .. (-3.42,4.34)
   .. controls (-2.80,4.07) and (2.45,4.04) .. (5.20,4.05);
 \draw[physical,->] (0.95,6.45)--(-0.15,6.43);
 \draw[physical,->] (-3.59,5.50)--(-3.59,4.98);
 \draw[physical,->] (-0.95,4.10)--(0.15,4.08);

 \node[anchor=west] at (-6.55,3.18)
   {(b) quadratic grazing in coordinates straightening \(X_p^-\), \(\beta(p)>0\)};
 \node[font=\scriptsize,text=violet!85!black] at (0,2.87)
   {violet physical orbit: \(X_p^-\longrightarrow X_p^+
      \longrightarrow X_p^-\)};
 \begin{scope}[shift={(0,0.62)},xscale=2.45,yscale=1.45]
  \fill[blue!3] (-2.15,-0.82) rectangle (2.15,1.18);
  \fill[orange!10]
    (-1.98,1.18)--(1.98,1.18)
    plot[smooth,domain=1.98:-1.98,samples=90] (\x,{0.30*\x*\x})--cycle;
  \draw[thick]
    plot[smooth,domain=-1.98:1.98,samples=90] (\x,{0.30*\x*\x});
  \node[anchor=east] at (-1.56,0.90) {\(\Sigma_p\)};
  \draw[thin] (-1.52,0.84)--(-1.40,{0.30*1.40*1.40});
  \node[font=\scriptsize,fill=orange!10,inner sep=1pt] at (0,1.04)
    {plus side \(h_p>0\): physical field \(X_p^+\)};
  \node[font=\scriptsize,fill=blue!3,inner sep=1pt] at (0,-0.67)
    {minus side \(h_p<0\): physical field \(X_p^-\)};

  \draw[section] (-2.15,-0.82)--(-2.15,1.05);
  \draw[section] (2.15,-0.82)--(2.15,1.05);
  \node[above left=2pt] at (-2.15,1.05) {\(\Gamma_{\rm in}\)};
  \node[above right=2pt] at (2.15,1.05) {\(\Gamma_{\rm out}\)};

  \draw[blue!65!black,thick,dash dot] (-2.15,-0.43)--(2.15,-0.43);
  \draw[blue!65!black,thick,dash dot,->] (1.55,-0.43)--(2.12,-0.43);
  \node[left=3pt,blue!65!black,font=\scriptsize] at (-2.15,-0.43)
    {\(I=q_-<0\) (no hit)};

  \draw[black,thick] (-2.15,0)--(2.15,0);
  \draw[black,thick,->] (1.55,0)--(2.12,0);
  \node[left=3pt,font=\scriptsize] at (-2.15,0) {\(I=0\) (tangent)};
  \fill (0,0) circle (1.15pt);
  \node[below=4pt] at (0,0) {\(z_p\)};

  \draw[singular] (-2.15,0.36)--(2.15,0.36);
  \draw[singular,->] (1.48,0.36)--(2.10,0.36);
  \node[left=3pt,font=\scriptsize] at (-2.15,0.36)
    {\(I=q>0\) (crossing)};
  \coordinate (ein) at ({-sqrt(0.36/0.30)},0.36);
  \coordinate (eref) at ({sqrt(0.36/0.30)},0.36);
  \coordinate (ephys) at ({sqrt(0.56/0.30)},0.56);
  \draw[physical] (-2.15,0.36)--(ein);
  \draw[physical,postaction={decorate},
    decoration={markings,mark=at position .55 with {\arrow{>}}}]
    (ein) .. controls (-0.60,0.36) and (0.88,0.56) .. (ephys);
  \draw[physical] (ephys)--(2.15,0.56);
  \draw[physical,->] (-1.90,0.36)--(-1.48,0.36);
  \draw[physical,->] (1.60,0.56)--(2.08,0.56);

  \fill[violet!85!black] (ein) circle (1.1pt);
  \node[above left=2pt,text=violet!85!black] at (ein)
    {\(e_{\rm in}\)};
  \draw[gray!75,fill=white,thick] (eref) circle (1.15pt);
  \node[below=3pt,gray!75!black] at (eref)
    {\(e_{\rm ref}\)};
  \fill[violet!85!black] (ephys) circle (1.1pt);
  \node[above=3pt,text=violet!85!black] at (ephys)
    {\(e_{\rm phys}\)};

  \draw[violet!85!black,thick,->] (2.15,0.36)--(2.15,0.56);
  \node[right=3pt,text=violet!85!black] at (2.15,0.46)
    {\(\mathcal D(q)>0\)};

  \draw[->] (-1.90,-1.17)--(-1.22,-1.17) node[right] {\(t\)};
  \draw[->] (-1.90,-1.17)--(-1.90,-0.70) node[above] {\(I\)};
 \end{scope}
\end{tikzpicture}
\caption{Schematic, not-to-scale geometry; no numerical trajectories are
used.  Panel (a) shows the singular entry--exit passage and its
positive-\(\eps\) shadow.  In (b), \(X_p^-=\partial_t\) and
\(h_p(t,I)=I-\kappa(p)t^2/2+\text{h.o.t.}\); hence the orange upper region is
the \(X_p^+\) side and the blue lower region is the \(X_p^-\) side.  The
\(q\)-labels denote incoming
\(I\)-levels, not regions: \(q_-<0\) misses \(\Sigma_p\), \(I=0\) is tangent,
and \(q>0\) crosses twice.  Its gray dashed line is the artificial \(X_p^-\)
reference.  For \(\beta(p)>0\), the violet solid physical orbit follows
\(X_p^-\to X_p^+\to X_p^-\) and exits at
\(I=q+\mathcal D(q)\), with \(\mathcal D(q)>0\) of order \(q^{3/2}\).}
\label{fig:geometry}
\end{figure}

The selected first-return circuit has the fixed oriented itinerary below;
the separation hypothesis makes the entire post-grazing leg regular:
\[
\begin{aligned}
 \Pi&\longrightarrow \text{entry--exit block}
 \longrightarrow \text{regular tube}
 \longrightarrow \text{grazing block}\\
 &\longrightarrow \text{regular return tube}\longrightarrow\Pi .
\end{aligned}
\tag{IT}\label{eq:itinerary}
\]

\paragraph{The scalar mechanism.}
Figure~\ref{fig:germs} plots explicit analytic toy functions from (1.1).
The curves are formula-generated illustrations, not ODE data or a numerical
premise.  The proof uses uniform one-sided curvature estimates and Rolle's
theorem.

\begin{figure}[!ht]
\centering
\begin{minipage}[t]{0.34\textwidth}
\centering
\begin{tikzpicture}[x=33cm,y=225cm,>=Latex,
                    line cap=round,line join=round]
 \draw[->] (-0.078,0)--(0.052,0) node[right] {\(q\)};
 \draw[->] (0,-0.0039)--(0,0.0108) node[above] {\(F\)};
 \begin{scope}
  \clip (-0.075,-0.0036) rectangle (0.048,0.0105);
  \draw[very thick,blue!70!black]
    plot[smooth,domain=-0.075:0,samples=90]
      (\x,{-0.003+\x*\x});
  \draw[very thick,blue!70!black]
    plot[smooth,domain=0:0.048,samples=90]
      (\x,{-0.003+\x*\x+\x*sqrt(\x)});
 \end{scope}
 \fill[blue!70!black] (-0.0547723,0) circle (1.2pt)
   node[above left=1pt,font=\scriptsize] {\(q_1\)};
 \fill[blue!70!black] (0.0190817,0) circle (1.2pt)
   node[above right=1pt,font=\scriptsize] {\(q_2\)};
 \node[diamond,draw=blue!70!black,fill=white,inner sep=1.1pt]
   at (0,-0.003) {};
\end{tikzpicture}

\smallskip
{\scriptsize\textbf{(a)} \(cd>0\): two local zeros.}
\end{minipage}\hfill
\begin{minipage}[t]{0.62\textwidth}
\centering
\begin{tikzpicture}[x=10.0cm,y=82cm,>=Latex,
                    line cap=round,line join=round]
 \fill[gray!10] (-0.155,-0.028) rectangle (0.12,0.016);
 \draw[densely dashed,gray!70] (-0.155,-0.0275) rectangle (0.12,0.0155);
 \node[font=\scriptsize,align=center,gray!70!black]
   at (-0.0175,0.0125) {local window};
 \draw[->] (-0.17,0)--(0.70,0) node[right] {\(q\)};
 \draw[->] (0,-0.029)--(0,0.017) node[above] {\(F\)};
 \begin{scope}
  \clip (-0.16,-0.028) rectangle (0.68,0.016);
  \draw[very thick,red!75!black]
    plot[smooth,domain=-0.155:-0.08,samples=70]
      (\x,{0.0032+0.16*\x+\x*\x});
  \draw[very thick,red!75!black]
    plot[smooth,domain=-0.08:0,samples=70]
      (\x,{0.0032+0.16*\x+\x*\x});
  \draw[very thick,red!75!black]
    plot[smooth,domain=0:0.0165839162,samples=40]
      (\x,{0.0032+0.16*\x+\x*\x-\x*sqrt(\x)});
  \draw[very thick,red!75!black]
    plot[smooth,domain=0.0165839162:0.12,samples=80]
      (\x,{0.0032+0.16*\x+\x*\x-\x*sqrt(\x)});
  \draw[very thick,gray!65,densely dashed]
    plot[smooth,domain=0.12:0.3859160838,samples=110]
      (\x,{0.0032+0.16*\x+\x*\x-\x*sqrt(\x)});
  \draw[very thick,gray!65,densely dashed]
    plot[smooth,domain=0.3859160838:0.68,samples=120]
      (\x,{0.0032+0.16*\x+\x*\x-\x*sqrt(\x)});
 \end{scope}
 \fill[red!75!black] (-0.136569,0) circle (1.0pt)
   node[above left=1pt,font=\scriptsize] {\(q_1\)};
 \fill[red!75!black] (-0.0234315,0) circle (1.0pt)
   node[below left=1pt,font=\scriptsize] {\(q_2\)};
 \fill[red!75!black] (0.0778035,0) circle (1.0pt)
   node[above right=1pt,font=\scriptsize] {\(q_3\)};
 \draw[gray!70!black,fill=white,line width=.6pt]
   (0.62624947,0) circle (1.5pt)
   node[above right=1pt,font=\scriptsize] {\(q_4\)};
 \node[font=\scriptsize,align=center,text=gray!70!black]
   at (0.43,0.0060) {toy-function zero only\\not a claimed cycle};
 \node[diamond,draw=red!75!black,fill=white,inner sep=1.0pt]
   at (-0.08,-0.0032) {};
 \node[diamond,draw=red!75!black,fill=white,inner sep=1.0pt]
   at (0.0165839162,0.0039928001) {};
 \node[diamond,draw=gray!65,fill=white,inner sep=1.0pt]
   at (0.3859160838,-0.0258615501) {};
\end{tikzpicture}

\smallskip
{\scriptsize\textbf{(b)} \(cd<0\): three local zeros; \(q_4\) is outside the theorem domain.}
\end{minipage}
\caption{Plots of the explicit analytic scalar toy functions
\(F_{a,b}(q)=a+bq+q^2+dq_+^{3/2}\).  In (a),
\((a,b,d)=(-3/1000,0,1)\) gives two zeros.  In (b),
\((a,b,d)=(2/625,4/25,-1)\) gives three filled zeros in the shaded
local window.  The gray dashed exterior branch and its open point
\(q_4\approx0.626249\) lie outside the local theorem domain; no periodic
orbit is claimed there.  The decimals approximate algebraic roots, the window
is illustrative rather than a certified theorem radius, and diamonds mark
turning points.  Thus the opposite-sign local three-zero bound is sharp but
not global.}
\label{fig:germs}
\end{figure}

\paragraph{A concrete realization.}
The scalar picture is not imposed by hand.  Section~7 verifies the
piecewise-polynomial family
\[
 X^-_{\eps,\mu}(x,y)=
 \begin{pmatrix}\eps-m_\mu(x)y^2\\ m_\mu(x)xy\end{pmatrix},
 \qquad
 X^+_{\eps,\mu}=X^-_{\eps,\mu}
   +\beta(y-1)\begin{pmatrix}0\\1\end{pmatrix},
\tag{1.3}
\]
where
\[
 m_\mu(x)=1+\alpha x(1-x^2)(x^2-3/7)
 +\mu_0x+\mu_1(x^3-3x/5).
\]
Here \(\mu=(\mu_0,\mu_1)\), the switching function is \(h=y-1\), and the
physical system uses \(X^-\) for \(y\le1\) and \(X^+\) for \(y\ge1\).
The line \(y=0\) is invariant and becomes the critical line at \(\eps=0\);
the base fast arc is the upper unit semicircle, tangent to \(y=1\) at the
grazing point \((0,1)\).  Section~7 computes the two unfolding directions as
\[
 \partial_{\mu_0}(a,b)(0)
   =\left(-\frac23,-\frac43-\frac{16\alpha}{21}\right),
 \qquad
 \partial_{\mu_1}(a,b)(0)=\left(0,-\frac45\right).
\]
Thus \(\mu_0\) moves the balance, or constant, coefficient (and also the
linear coefficient), whereas \(\mu_1\) preserves balance to first order and
changes the linear, or multiplier, coefficient.  These directions are
independent.
With suitable separated sections it has
\[
 c=\frac{8\alpha}{7},\qquad
 d=\frac{4\sqrt2}{3}\beta,\qquad
 \det D_\mu(a,b)(0)=\frac8{15}.
\tag{1.4}
\]
Thus the hypotheses are simultaneously realizable and the two signs can be
chosen independently.  After the rank-two perturbation by \(\mu\),
\(\alpha\beta>0\) realizes the sharp two-cycle sign class, whereas
\(\alpha\beta<0\) realizes the sharp three-cycle sign class.  The full
verification is postponed so that the general geometry can be stated in
intrinsic section-map language.

\FloatBarrier

\paragraph{Previous work and the theorem boundary.}
The closest theorem-level predecessor is Huang--Huzak--Yao (HHY)
\cite[Theorem~2.3]{HuangHuzakYao2026}.  Their part~(I) permits the tangent
label \(x_0=1\) only when \(\lambda(1)\ne0\), while the higher-cyclicity
parts~(II)--(III) assume \(x_0\ne1\).  The tangent-neutral interface is
therefore outside that theorem.  We treat it under separated
entry--exit/grazing hypotheses, which are not automatic for every HHY model.

General entry--exit criteria are developed in \cite{AiSadhu2020,AiYi2024};
common cycles, balanced canards, and small-death cycles are treated in
\cite{DeMaesschalckDumortierRoussarie2011,Dumortier2011SlowDivergence,
Huzak2018SmallDeath,Hsu2019,YaoHuangHuzak2024}.  Piecewise predator--prey
precedents include the Holling-I focus problem of Zegeling--Kooij
\cite{ZegelingKooij2020}, the concrete two-cycle model of Li--Wang--Wu
\cite{LiWangWu2021}, and Zegeling's regular--singular Li\'enard analysis with
a cutoff Gause application \cite{Zegeling2024}.  Respectively, their focus,
model-specific, and degenerate-canard mechanisms do not combine an
invariant-line delay, separated continuous quadratic grazing, and a neutral
return.

Recent piecewise-smooth biological models also produce three cycles or
cyclicity three: Zhang--Qiu--Cai--Shen obtain multiple relaxation cycles from
fold/corner geometry and canard or superexplosion mechanisms
\cite{ZhangQiuCaiShen2025}; Chen--Li--Tang use singular Hopf, canard, and
double-headed-canard mechanisms in a Leslie--Gower model
\cite{ChenLiTang2026}; and Zhu--Liu study crossing cycles around a boundary
focus with additive Allee effect \cite{ZhuLiu2026}.  Consequently, the mere
occurrence of three cycles in a piecewise-smooth predator--prey model is not
our novelty claim.  Those organizers differ from the present neutral
invariant-line entry--exit circuit with a quadratically tangent fast return;
our contribution is the sign-sharp local classification and rank-two
unfolding of that specific composite germ.

De Maesschalck--Schecter prove smooth dependence of the exact invariant-axis
entry--exit map \cite{DeMaesschalckSchecter2016}; Hsu supplies the
nonvertical layer geometry used here \cite{Hsu2019}, and Wang--Zhang analyze
smooth degenerate turning points \cite{WangZhang2018}.  These settings have
no moving switching threshold or itinerary-dependent \(q_+^{3/2}\)
correction.  Our passage retains both endpoint maps and the incoming balance
label; parameter dependence alone is not claimed as new.

The continuous quadratic-grazing \(3/2\) law is classical
\cite{DankowiczNordmark2000,diBernardoBuddChampneys2001}, as are
one-dimensional maps with such terms \cite{HalseHomerdiBernardo2003}.
Roberts treats canards generated by critical-manifold corners
\cite{Roberts2016}; recent work treats other grazing, sliding, higher-contact,
and regularized configurations
\cite{FangChen2025,ChenFangLi2026Arbitrary,ChenFangLi2026Symmetric,
DeMaesschalckHuzakPerez2026,HuangHuzakPerezYao2026}.  These mechanisms do not
supply the separated composition used here.

In the broader finite-cyclicity tradition, Mourtada treats generic hyperbolic
polycycles through normalized Dulac transitions \cite{Mourtada1991Algorithm},
the program of Dumortier, Roussarie, and Rousseau organizes quadratic
Hilbert-sixteenth finiteness through analytic graphics
\cite{DumortierRoussarieRousseau1994}, and Dumortier, Ilyashenko, and Rousseau
develop saddle-node normal forms for families of graphics
\cite{DumortierIlyashenkoRousseau2002}.  These are framework and
methodological background, not results subsuming our continuous
piecewise-smooth slow--fast class.  Their prepared Dulac/analytic
transitions are not the one-sided moving-threshold Puiseux germ derived
here, and we prove no formal embedding or transfer theorem between the two
settings.  Our theorem neither settles a DRR quadratic graphic nor supplies
a bound for \(H(2)\).

Once the ramified return germ is available, its scalar zero count is
elementary and is not claimed as new in isolation.  To the best of our
knowledge, the new contribution is the geometric derivation of that germ,
parameter-uniform derivative control, physical realization, and sign-sharp
positive-\(\eps\) cyclicity for the separated entry--exit/grazing circuit at
unit multiplier.

\paragraph{Contributions.}
The paper proves four linked results.
\begin{enumerate}
\item We establish the finite-order, parameter-uniform physical passage in
the Hsu/HHY nonvertical-fiber form, with explicit source and target spaces for
the endpoint-footpoint maps and a free incoming balance label.
\item In the exact moving penetration coordinate, we compose this passage
with the continuous-grazing transition and regular tubes, obtaining a
derivative-controlled ramified return germ with an explicit geometric
coefficient.
\item Uniform ordinary and ramified curvature estimates give the sharp bound
\[
 \Cyc\le
 \begin{cases}
 2,&cd>0,\\
 3,&cd<0,
 \end{cases}
\tag{1.5}
\]
and a rank-two unfolding attains the corresponding bound in each sign class
for every sufficiently small fixed positive \(\eps\).
\item An explicit polynomial benchmark realizes both sign cases.  A
Gause/Holling-type application starts from the equations and cutoff-response
class introduced by Kooij--Zegeling \cite{KooijZegeling2019}; the particular
base member selected here, its quartic response coefficient \(\varrho\), and
the codimension-two parameter choice are explicitly new modelling assumptions.
We determine \(\sgn(cd)=-\sgn\varrho\), transfer the
biological unfolding to the canonical return coefficients, and validate
explicit nonzero singular parameter points at \(\varrho=\pm1/20\).
\end{enumerate}

\paragraph{Proof architecture and evidence.}
Sections~3--6 follow the logical chain from the two local passages, through
their moving-coordinate composition, to the one-sided curvature estimates
and sharp zero count.  The rank-two hypothesis (UNF) is used only for
attainment.  The polynomial benchmark, smooth-reference curvature transfer,
and Gause realization then verify joint realizability and the model-specific
hypotheses.  The general cyclicity statements are analytic.  Two independent
interval backends certify the base-point signs; validated \(C^1\)
Poincar\'e maps and interval Newton certify the two explicit Gause parameter
points.
The positive-\(\eps\) biological return plots use ordinary floating-point
integration and are evidence rather than proof.

The theorem is not a recognition result for an arbitrary continuous
piecewise-smooth system.  It assumes an exact invariant-line entry--exit
block, separated singular boxes, continuous quadratic grazing, and regular
global tubes.  It gives a local count in a selected return neighborhood, not
a global limit-cycle bound or an empirical ecological calibration.

\section{Separated tangent entry--exit configurations}

Let \(p=(\eps,\mu)\), where
\[
 \eps\in[0,\eps_0),\qquad \mu\in M\subset\R^2,
\]
and \(M\) is a neighborhood of \(0\).  All neighborhoods below may be
shrunk simultaneously.  The branch fields, sections, and coordinate
changes are assumed \(C^\infty\) jointly in their displayed variables.
Only finite orders are used.  Throughout,
\(q_+=\max\{q,0\}\).

The hypotheses are grouped by geometric role rather than presented as an
undifferentiated list.  The labels (EE), (GR), (REG), (BAL), and (UNF) will
also identify the proof module that uses each assumption.

\subsection{Named geometric hypotheses}

\begin{hypothesis}[(EE) Invariant-line entry--exit block]
\label{hyp:ee}

After a possible common reversal of time, the field in the entry--exit box
has the form
\[
 \dot u=u f(u,v,\eps,\mu),\qquad
 \dot v=\eps g(u,v,\eps,\mu)+u h_0(u,v,\eps,\mu).
\tag{2.1}
\]
There are \(v_-<0<v_+\) and \(\gamma>0\) such that, uniformly for small
\(\mu\),
\[
\begin{aligned}
 g(0,v,0,\mu)&\le-\gamma,
       &&v_-\le v\le v_+,\\
 f(0,v,0,\mu)&>0,&&v_- \le v<0,\\
 f(0,v,0,\mu)&<0,&&0<v\le v_+ .
\end{aligned}
\tag{2.2}
\]
The last two signs are uniform on compact subintervals away from \(0\).
Choose
\[
 \widetilde v_0^*>0>\widetilde v_1^*,\qquad
 \int_{\widetilde v_0^*}^{\widetilde v_1^*}
 \frac{f(0,v,0,0)}{g(0,v,0,0)}\,\dd v=0.
\tag{2.3}
\]
Writing \(F(v,\mu)=f(0,v,0,\mu)/g(0,v,0,\mu)\), choose disjoint
critical-line neighborhoods
\(U_0^{\rm c}\ni\widetilde v_0^*\) and
\(U_1^{\rm c}\ni\widetilde v_1^*\).  For \(\mu\) near \(0\), let
\(\mathcal B_\mu:U_0^{\rm c}\to U_1^{\rm c}\) be the unique
opposite-endpoint germ
satisfying
\[
 \int_{\zeta}^{\mathcal B_\mu(\zeta)}F(v,\mu)\,\dd v=0.
\tag{2.4}
\]
It is normalized by
\(\mathcal B_0(\widetilde v_0^*)=\widetilde v_1^*\).  The disjoint source
and target neighborhoods exclude the trivial solution
\(\mathcal B_\mu(\zeta)=\zeta\).
Fix \(u_1>0\) small.  The layer fibers at the two endpoints are
\[
 \partial_u\mathscr V_i
 =\frac{h_0(u,\mathscr V_i,0,\mu)}
        {f(u,\mathscr V_i,0,\mu)},\qquad
 \mathscr V_i(0;\zeta,\mu)=\zeta,\qquad
 \Lambda_i(\zeta,\mu)=\mathscr V_i(u_1;\zeta,\mu).
\tag{2.5}
\]
Put \(v_i^*=\Lambda_i(\widetilde v_i^*,0)\), and choose disjoint
physical-section neighborhoods \(V_i\ni v_i^*\).  The
parameter-preserving maps
\[
 \boldsymbol\Lambda_i:(\zeta,\mu)\longmapsto
   (\Lambda_i(\zeta,\mu),\mu),\qquad
 \boldsymbol\Lambda_i^{-1}:(v,\mu)\longmapsto
   (\pi_i(v,\mu),\mu).
\]
are mutually inverse local-diffeomorphism germs at
\((\widetilde v_i^*,0)\) and \((v_i^*,0)\), respectively.  After shrinking
and relabelling \(M\), take fixed \(V_i\) so that \(V_i\times M\) lies in
one image representative of \(\boldsymbol\Lambda_i\) and
\(\pi_i(V_i,\mu)\Subset U_i^{\rm c}\).  Shrink \(V_0\) once more so that
the balance-map output lies in the domain of
\(\Lambda_1(\cdot,\mu)\) and the composite lands in \(V_1\), uniformly in
the retained parameters.  No identity
\(\Lambda_i(U_i^{\rm c},\mu)=V_i\) is asserted.  On these common inner
representatives, the singular physical transition is a germ
\(V_0\to V_1\):
\[
 \mathcal T_{\rm ee}^0(\cdot,\mu)
 =\Lambda_1(\cdot,\mu)\circ\mathcal B_\mu
   \circ\pi_0(\cdot,\mu)
 =\pi_1(\cdot,\mu)^{-1}\circ\mathcal B_\mu
   \circ\pi_0(\cdot,\mu).
\tag{2.6}
\]
In particular, it is generally not the bare balance map.
\end{hypothesis}

\begin{hypothesis}[(GR) Continuous quadratic grazing]
\label{hyp:gr}

In a disjoint box let \(h_p\) be a switching function.  The physical field
is \(X_p^-\) on \(h_p\le0\) and \(X_p^+\) on \(h_p\ge0\), where
\(X_p^\pm\) are smooth extensions satisfying
\[
 X_p^+=X_p^-\quad\text{on }\Sigma_p=\{h_p=0\}.
\tag{2.7}
\]
There is a smooth family \(z_p\in\Sigma_p\) for which
\[
\begin{gathered}
 \|d_zh_p(z_p)\|\ge\nu_0>0,\qquad
 X_p^-(z_p)\ne0,\qquad X_p^-h_p(z_p)=0,\\
 \kappa(p):=-(X_p^-)^2h_p(z_p)\ge\kappa_0>0.
\end{gathered}
\tag{2.8}
\]
Hadamard factorization gives \(X_p^+-X_p^-=h_pA_p\).  Put
\[
 \beta(p)=d_zh_p\bigl(A_p(z_p)\bigr),\qquad \beta(0)\ne0.
\tag{2.9}
\]

Let \(I_p\) be the local first integral of \(X_p^-\), normalized by
\[
 X_p^-I_p=0,\qquad I_p(z_p)=0,\qquad dI_p(z_p)=dh_p(z_p).
\tag{2.10}
\]
\end{hypothesis}

\begin{hypothesis}[(REG) Separated regular return geometry]
\label{hyp:reg}
Choose a local return section \(\Pi\), with coordinate \(\eta=0\) on the
base itinerary, so that the cyclic order is \eqref{eq:itinerary}.  Every
piece outside the two singular boxes is a uniformly regular flow tube: the
field is nonzero, flight times are bounded, and all section transversality
constants are bounded away from zero.  Every nearby orbit meets \(\Pi\)
exactly once per selected circuit.

Pulling \(I_p\) back to \(\Pi\) through the preceding modules defines
\[
 Q_p(\eta)=\rho(\eta,p),\qquad
 \rho(0,0)=0,\qquad \partial_\eta\rho(0,0)\ne0.
\tag{2.11}
\]
Thus \(q=Q_p(\eta)\) is the exact signed penetration, and \(q=0\) is the
moving grazing graph.  For \(\eps>0\), let \(P_p^{\rm phys}\) be the
physical Poincar\'e return and let \(P_p^{{\rm ref},{\rm phys}}\) be the
corresponding return obtained by using \(X_p^-\) throughout the grazing
box.  Theorem~\ref{thm:entryexit} below proves the parameter-uniform
entry--exit extension; composition with the regular connecting modules then
extends these returns to \(\eps=0\).  This extension is a conclusion of that
theorem, not an additional part of \textup{(REG)}.  Denote the resulting
families on \([0,\eps_0)\times M\) by \(\mathcal R_p\) and
\(\mathcal R_p^{\rm ref}\).  Thus
\(\mathcal R_p=P_p^{\rm phys}\) and
\(\mathcal R_p^{\rm ref}=P_p^{{\rm ref},{\rm phys}}\) for \(\eps>0\),
whereas their \(\eps=0\) members are singular parameter extensions, not
finite-time physical Poincar\'e returns.  No closure or multiplier condition
is imposed in \textup{(REG)}.  Set
\[
\begin{aligned}
 \widehat{\mathcal R}_p&=Q_p\circ \mathcal R_p\circ Q_p^{-1},&
 \Delta(q,p)&=\widehat{\mathcal R}_p(q)-q,\\
 \widehat{\mathcal R}_p^{\rm ref}
 &=Q_p\circ \mathcal R_p^{\rm ref}\circ Q_p^{-1},&
 S(q,p)&=\widehat{\mathcal R}_p^{\rm ref}(q)-q.
\end{aligned}
\tag{2.12}
\]
Thus \(q\) is the exact moving penetration coordinate, \(S\) is the smooth
reference displacement, and \(\Delta-S\) is the physical correction created by
switching through the grazing box.
\end{hypothesis}

\begin{hypothesis}[(BAL) Simple balanced neutral return]
\label{hyp:bal}
Assume that the reference return closes at the base point and is neutral
there, with nonzero smooth curvature:
\[
 S(0,0)=S_q(0,0)=0,\qquad
 c:=\frac12S_{qq}(0,0)\ne0.
\tag{2.13}
\]
\end{hypothesis}

\begin{hypothesis}[(UNF) Rank-two unfolding]
\label{hyp:unf}
Put \(a(p)=S(0,p)\), \(b(p)=S_q(0,p)\), and impose the rank condition
\[
 \det\!\left[
 \partial_\mu(a,b)
 \right]_{(\eps,\mu)=(0,0)}\ne0.
\tag{2.14}
\]
\end{hypothesis}

\begin{definition}[Separated tangent entry--exit configuration]
\label{def:configuration}
A return circuit satisfying \textup{(EE)}, \textup{(GR)}, and
\textup{(REG)} is a \emph{separated tangent entry--exit configuration}.  It
is \emph{simple balanced neutral} when \textup{(BAL)} holds and
\emph{rank-two versal} when \textup{(UNF)} also holds.
\end{definition}

\begin{definition}[Local cyclicity in the selected family]
\label{def:cyclicity}
Let \(\Delta(q,p)\) be the displacement germ in \((2.12)\).  Its local
scalar cyclicity at \((q,p)=(0,0)\) is
\[
 \Cyc_{\rm scal}:=
 \sup\left\{N\in\mathbb N:
 \begin{array}{l}
  \text{there are }p_n\to0\text{ and distinct zeros }
  q_{n,1},\ldots,q_{n,N}\text{ of }\Delta(\cdot,p_n),\\
  \max_{1\le j\le N}|q_{n,j}|\longrightarrow0
 \end{array}\right\}.
\]
The physical local cyclicity \(\Cyc_{\rm phys}\) is defined by additionally
requiring the singular component of every \(p_n\) to satisfy \(\eps_n>0\).
The supremum of an empty admissible collection is taken to be zero.  Zeros
are counted without multiplicity.  Under \textup{(REG)}, the physical zeros
are in bijection with limit cycles in the selected return neighborhood.  The
singular \(\eps=0\) extension enters \(\Cyc_{\rm scal}\) but is not itself
interpreted as a periodic orbit.
\end{definition}

\begin{center}
\begin{tabular}{@{}lll@{}}
\toprule
Package & Geometric content & Main use\\
\midrule
(EE) & invariant line, balance, endpoint fibers & physical passage\\
(GR) & continuous quadratic contact, \(\beta\ne0\) & \(3/2\) transition\\
(REG) & separated boxes, moving \(q\), regular tubes & composition\\
(BAL) & closure, unit multiplier, \(c\ne0\) & zero count\\
(UNF) & rank two in \((a,b)\) & sharp attainment\\
\bottomrule
\end{tabular}
\end{center}

\begin{theorem}[Moving-threshold return and sharp cyclicity]
\label{thm:main}
Fix a finite integer \(\ell\ge2\).
Assume \textup{(EE)}, \textup{(GR)}, \textup{(REG)}, and \textup{(BAL)}.
\emph{(i) Return germ.}
There are a section neighborhood \(U_\ell\) and a parameter neighborhood
\(P_\ell\), which may depend on \(\ell\), such that, for
\((q,p)\in U_\ell\times P_\ell\),
\[
 \Delta(q,p)=S(q,p)+q_+^{3/2}
 K(\sqrt{q_+},q,p).
\tag{2.15}
\]
Here \(S\) is jointly \(C^{\ell+1}\) in \((q,p)\), while \(K(s,q,p)\) is
jointly \(C^\ell\) for \(s\ge0\).  Moreover,
\[
 d:=K(0,0,0)
 =L_0\frac{4\sqrt2}{3\sqrt{\kappa(0)}}\,\beta(0)\ne0,
\tag{2.16}
\]
where \(L_0\) is the derivative of the regular post-grazing transition from
the normalized outgoing \(I_0\)-coordinate to the conjugated \(q\)-output
coordinate.

\emph{(ii) Uniform upper bound.}
There are \(\delta>0\) and a parameter neighborhood \(P_0\) such that
\[
 \#\{q\in(-\delta,\delta):\Delta(q,p)=0\}
 \le
 \begin{cases}
 2,&cd>0,\\
 3,&cd<0,
 \end{cases}
\tag{2.17}
\]
for every \(p\in P_0\).

\emph{(iii) Rank-two unfolding and sharpness.}
If, in addition, \textup{(UNF)} holds, there are
\(\eps_1>0\) and a
\(C^\ell\) curve
\(\mu_*:[0,\eps_1)\to M\) such that
\[
 \mu_*(0)=0,\qquad
 \mu_*(\eps)=O(\eps),\qquad
 a(\eps,\mu_*(\eps))=b(\eps,\mu_*(\eps))=0.
\tag{2.18}
\]
The applicable bound is sharp in each sign class: for every sufficiently
small fixed \(\eps>0\), parameters arbitrarily close to \(\mu_*(\eps)\)
realize two simple fixed points when \(cd>0\) and three when \(cd<0\).  Along
the realizing paths, the fixed points tend to \(q=0\) and
\(\mu\to\mu_*(\eps)\); jointly,
\((\eps,\mu_*(\eps))\to(0,0)\) as \(\eps\downarrow0\).

\emph{(iv) Physical interpretation.}
For \(\eps>0\), the zeros counted in \((2.17)\), including the simple
zeros on the realizing paths, are fixed points of the conjugated physical
Poincar\'e return and are in bijection with distinct limit cycles in the
selected return neighborhood.  At \(\eps=0\), \((2.17)\) is a zero count
for the singular extension only; no periodic-orbit correspondence is
asserted there.  Consequently,
\[
 \Cyc_{\rm phys}\le
 \begin{cases}
  2,&cd>0,\\
  3,&cd<0,
 \end{cases}
\]
and equality holds in the indicated sign case when \textup{(UNF)} is
satisfied.
\end{theorem}

\begin{remark}[Finite regularity]
\label{rem:regularity}
The \(C^\infty\) assumption avoids a distracting optimization.  The
entry--exit result below gives a joint \(C^r\) physical passage from
\(C^{r+4}\) data.  To obtain a \(C^\ell\) coefficient \(K\) after smooth
postcomposition, we invoke that result with \(r=\ell+1\).  The smooth
reference displacement retains \(C^{\ell+1}\) regularity; the one
derivative loss occurs only in the coefficient \(K\).  In particular, the
curvature proof of \((2.17)\) uses a \(C^3\) passage and is covered by
\(C^7\) entry--exit data.
\end{remark}

\section{Uniform physical entry--exit passage}

We isolate the only slow singularity used in Theorem~\ref{thm:main}.
De Maesschalck--Schecter, Corollary~1.2 and Remark~1
\cite{DeMaesschalckSchecter2016}, establish smooth dependence on the input,
\(\eps\), and additional finite-dimensional parameters for the exact
invariant-axis form.  Hsu's Theorem~5.1 \cite{Hsu2019} treats the physical
passage with nonvertical layer fibers for fixed coefficients.  What is
needed here is their interface: a finite-order parameter-uniform passage in
the Hsu/HHY form, including both endpoint layer maps while retaining the
incoming endpoint as a free phase label.  Parameter dependence alone is not
presented as a new phenomenon.

\begin{theorem}[Uniform physical entry--exit passage]
\label{thm:entryexit}
Fix \(r\ge1\).  Suppose the coefficients in \((2.1)\) are \(C^{r+4}\) on
a neighborhood of the compact passage set and satisfy \((2.2)\)--\((2.5)\).
Then, after shrinking the neighborhoods, the forward local passage from the
entry component of \(u=u_1\) to its exit component is defined for
\[
 v\in V_0,\qquad 0<\eps<\eps_1,\qquad |\mu|<\mu_1.
\]
Its exit coordinate
\(\mathcal T_{\rm ee}(v,\eps,\mu)\in V_1\) extends jointly \(C^r\) to
\(\eps=0\), and its singular value is exactly \((2.6)\).  Thus
\(\mathcal T_{\rm ee}:V_0\to V_1\) in the phase variable.
The word ``passage'' refers to the first exit inside the chosen isolating
entry--exit neighborhood and makes no assertion about later global
intersections.
\end{theorem}

The proof follows three steps.  First we straighten the endpoint layer
fibers and augment the system by exponential, clock, endpoint, and
coefficient labels.  Second we apply the frozen-label exchange result
separately at the attracting and repelling endpoints.  Third we transport
the two exchanged graph families to a common middle chart, verify the two
unresolved balance--flight equations, and apply
Lemma~\ref{lem:middle-clock}; only after that matching do we project to the
physical variables and normalize the auxiliary clock.

The next two results only expose the interface with Schecter's preparation
and General Exchange Lemma.  The corollary is a specialization of
\cite{Schecter2008}, not a new independent exchange lemma.

\begin{lemma}[Frozen-slice preservation]
\label{lem:frozen-slices}
Let a local augmented flow have coordinates \((w,\theta,\mu)\), where
\(\dot\theta=0\) and \(\dot\mu=0\).  Every stable or unstable
asymptotic-phase fiber is contained in one slice
\((\theta,\mu)=\mathrm{const}\).  Consequently every phase projection
preserves the frozen labels.  If, in addition, the reduced vector field is
nonzero on a compact tube, parameter-dependent reduced flow-box coordinates
may be chosen over the identity on those labels.
\end{lemma}

\begin{proof}
The \((\theta,\mu)\)-difference of two augmented trajectories is constant.
Two points on one asymptotic-phase fiber have trajectories whose distance
tends to zero in the relevant time direction, so that constant difference
must vanish.  The invariant fibers therefore lie in frozen slices, so the
phase projection is fiberwise over the identity on the frozen labels.  Under
the stated nonvanishing condition, the parameter-dependent flow-box theorem
applied in each such slice gives the final assertion.
\end{proof}

\begin{corollary}[Schecter exchange with frozen endpoint and coefficient labels]
\label{cor:augmented-exchange}
Fix \(r\ge1\), a compact frozen-label block, and one incoming and one
outgoing endpoint tube.  Suppose that, after a layer-footpoint change, the
augmented endpoint fields are joint \(C^{r+3}\) standard slow--fast systems
\[
 \dot n=n\,\widetilde{\mathsf G}(n,\widetilde z,\eps),
 \qquad
 \dot{\widetilde z}=\eps\widetilde{\mathsf K}
       (n,\widetilde z,\eps),
 \qquad
 \widetilde z=(\alpha,\xi,\tau,\theta,\mu),
\]
with \(\dot\theta=\dot\mu=0\).  Assume that the sole normal direction is
uniformly attracting at the incoming tube and uniformly repelling at the
outgoing tube, and use reversed time at the latter.  Assume also that the
reduced center vector is uniformly nonzero on both compact tubes.  Let the
two input manifolds be joint \(C^{r+1}\) graph families parameterized by
\((\tau_0,\theta,\mu)\) and \((s_{\rm out},\theta,\mu)\), respectively.
Suppose their
limiting center projections are immersions and that the appropriate reduced
center vector is nowhere tangent to either projected family.

If \(d_\mu=\dim\mu\), then Schecter's center and incoming-manifold
dimensions are
\[
 m_{\rm S}=4+d_\mu,\qquad p_{\rm S}=2+d_\mu,
 \qquad m_{\rm S}-p_{\rm S}-1=1.
\]
After one common shrinking, Schecter's prepared charts are joint
\(C^{r+1}\), preserve \((\theta,\mu)\)-slices, and his Theorem~3.1 produces,
in two separate applications, a forward incoming exchanged family and a
time-reversed outgoing exchanged family.  Both extend jointly \(C^r\) in all
ending coordinates, \(\eps\), \(\theta\), and \(\mu\).  The last displayed
dimension is the single unmatched center coordinate used for the
entry--exit balance.
\end{corollary}

\begin{proof}
Schecter's Sections~2.1, 2.4, and 2.5 and equations (2.9)--(2.16)
\cite{Schecter2008} prepare each standard endpoint system as
\[
\begin{aligned}
 \dot x&=A(x,y,z,\eps)x,\\
 \dot y&=B(x,y,z,\eps)y,\\
 \dot z&=\eps\bigl(e_1+L(x,y,z,\eps)xy\bigr).
\end{aligned}
\]
Here \(x\) and \(y\) are the stable and unstable normal blocks and \(z\)
is the prepared image of \(\widetilde z\).  At the incoming endpoint
\((\dim x,\dim y)=(1,0)\); at the outgoing endpoint
\((\dim x,\dim y)=(0,1)\), or equivalently \((1,0)\) after time reversal.
Thus \(xy\equiv0\).  In Schecter's coordinates adapted to the projected
incoming manifold, \(u_0\) is the rectified reduced-flight coordinate,
\(v_0\) contains its tangent and frozen-label coordinates, and the remaining
scalar \(w_0\) is the balance direction.  The center equations reduce to
\[
 \dot u_0=\eps,\qquad \dot v_0=0,
 \qquad \dot w_0=0.
\]
This verifies the rectified-speed condition with transit exponent one.
Lemma~\ref{lem:frozen-slices} shows that the preparation retains
\((\theta,\mu)\) among the center labels.

The finite-regularity remark following Schecter's Theorem~2.1 (printed
p.~413), together with the adapted construction in Sections~2.3--2.4,
gives a \(C^{r+1}\) prepared form from a \(C^{r+3}\) standard field.  Fix
the compact frozen-label block before choosing charts.  If the scalar stable
rates, using time reversal at the second endpoint, are at most
\(-\nu_{\rm n}<0\), take \(\lambda_{\rm S}=-\nu_{\rm n}/2\) and
\(\mu_{\rm S}=\nu_{\rm n}/4\).  In Schecter's rectified prepared center
coordinates the center solution operator is the identity, so
\(\beta_0>0\) may be chosen small enough that
\[
 \lambda_{\rm S}+\mu_{\rm S}+r\beta_0<0
 <\mu_{\rm S}-\max\{6,2r+1\}\beta_0.
\]
This verifies Schecter's \textup{(E1)}--\textup{(E2)} uniformly.  The stated
graph and flow-box assumptions give \textup{(E3)}--\textup{(E10)}.  Choose
\(\beta_0<\beta<\beta_1\) within the same strict inequalities and shrink the
endpoint tubes once.  Then the common transit bounds
\(K_1/\eps\le T\le K_2/\eps\), with \(K_2/K_1\) sufficiently close to one,
give \textup{(E11)} for every frozen label, exactly as in Schecter's
Section~4.1.
Schecter's Theorem~3.1 therefore gives \(C^r\) exchanged graph functions in
all ending chart variables and \(\eps\).  Frozen-slice preservation makes
this precisely joint \(C^r\) dependence on
\((\theta,\eps,\mu)\), as asserted.
\end{proof}

\begin{lemma}[From exchanged graphs to fixed-section traces]
\label{lem:exchange-to-trace}
Let \(\mathcal E_\eps(\eta,\lambda)\) be either exchanged graph family
provided by Corollary~\ref{cor:augmented-exchange}, where \(\eta\) denotes
its free endpoint variable and \(\lambda\) collects the frozen labels.
Suppose that its limiting reduced orbit meets an ending section
\(\Sigma_{\rm e}=\{\zeta=\zeta_{\rm e}\}\) transversely, with
\(
 |X_0\zeta|\ge\gamma>0
\)
on a compact set of such intersections.  Suppose also that the regular
orbit segment from \(\Sigma_{\rm e}\) to a fixed middle section
\(\Sigma_{\rm m}\) remains in a common flow box and meets
\(\Sigma_{\rm m}\) transversely.

After one common shrinking, every orbit in \(\mathcal E_\eps\) has a
unique first hit on \(\Sigma_{\rm e}\).  Its hitting time and endpoint are
joint \(C^r\) functions of \((\eta,\eps,\lambda)\), and regular transport
from that endpoint gives a joint \(C^r\) embedded trace on
\(\Sigma_{\rm m}\).  The same conclusion holds for an outgoing family
after applying it to the reversed field.
\end{lemma}

\begin{proof}
Let \(\varphi\) be a defining function for \(\Sigma_{\rm e}\) and let
\(\phi_\eps^t\) denote the endpoint flow in the prepared physical
coordinates.  At a limiting hit,
\[
 \partial_t\bigl(\varphi(\phi_0^t(z))\bigr)
 =X_0\varphi\ne0.
\]
The parameterized implicit-function theorem therefore gives a joint
\(C^r\) hitting time and hit point.  Compactness supplies one neighborhood
and the lower bound \(\gamma/2\); an isolating strip on the incoming side
of \(\Sigma_{\rm e}\) makes this hit the unique first one.  The regular
segment is treated by the same hitting-time argument in its common flow
box.  Flow maps between transverse sections are local diffeomorphisms, so
the transported family remains an embedding.  Replacing \(X_\eps\) by
\(-X_\eps\) proves the outgoing assertion.
\end{proof}

\begin{lemma}[Uniform exponential flatness]
\label{lem:exp-flatness}
Fix \(\nu>0\), a compact interval \(I_\xi\subset[\nu,\infty)\),
and nonnegative integers \(m\) and \(r\).  On \(I_\xi\), set
\[
 E_m(\xi,\eps)=
 \begin{cases}
  \eps^{-m}e^{-\xi/\eps},&\eps>0,\\
  0,&\eps=0.
 \end{cases}
\]
Then \(E_m\) is \(C^r\) up to \(\eps=0\), and every derivative containing
an \(\eps=0\) value vanishes there.  More precisely, for \(j+k\le r\),
\[
 \left|\partial_\eps^j\partial_\xi^kE_m(\xi,\eps)\right|
 \le C_{jkm}\eps^{-(m+2j+k)}e^{-\nu/\eps}.
\tag{3.0a}
\]
The same conclusion holds, including all mixed derivatives in \(y\) and
\(\eps\), after composition with a joint \(C^r\) function
\(\Xi(y,\eps)\ge\nu\).
\end{lemma}

\begin{proof}
Repeated differentiation produces \(e^{-\xi/\eps}\) times a finite sum of
bounded powers of \(\xi\) and negative powers of \(\eps\); the right-hand
side of \((3.0\mathrm a)\) is a harmless common majorant.  Exponential
decay dominates every negative power of \(\eps\).  The chain rule gives the
composition statement.
\end{proof}

\begin{lemma}[Explicit matching on a physical middle section]
\label{lem:middle-clock}
Fix a compact block of frozen labels \((\theta,\mu)\) and a common middle
section \(\Sigma_{\rm m}=\{s=s_{\rm m}\}\).  Let \(Z_{\rm m}\subset
\mathbb R^2\) have coordinates \((\xi,\tau)\), and put
\(R(\xi,\tau)=(\xi,-\tau)\).  Suppose the forward incoming and backward
outgoing exchanged traces on \(\Sigma_{\rm m}\) are joint \(C^r\)
embeddings
\[
\begin{aligned}
 \Gamma^-_\eps(\tau_0,\theta,\mu)
   &=(\Xi^-_\eps,T^-_\eps)\in Z_{\rm m},\\
 \Gamma^+_\eps(s_{\rm out},\theta,\mu)
   &=(\Xi^+_\eps,T^+_\eps)\in Z_{\rm m},
\end{aligned}
\tag{3.0b}
\]
and lie where \(\xi\ge\nu>0\) on the invariant physical sheet
\[
 n=N(\xi,\eps):=
 \begin{cases}e^{-\xi/\eps},&\eps>0,\\0,&\eps=0.\end{cases}
\tag{3.0c}
\]
For an explicit common-point unknown \(\mathbf z_{\rm m}\in RZ_{\rm m}\),
define
\[
 \mathfrak M_\eps
 (\mathbf z_{\rm m},s_{\rm out},\tau_0;\theta,\mu)
 =
 \begin{pmatrix}
  \mathbf z_{\rm m}-R\Gamma^-_\eps(\tau_0,\theta,\mu)\\
  R\Gamma^+_\eps(s_{\rm out},\theta,\mu)
   -R\Gamma^-_\eps(\tau_0,\theta,\mu)
 \end{pmatrix}\in\mathbb R^4.
\tag{3.0d}
\]
Thus, for fixed \((\theta,\eps,\mu)\), its four unknowns are
\(\dim\mathbf z_{\rm m}+1+1=4\).

Suppose that, as functions on fixed neighborhoods,
\[
\begin{aligned}
 \Gamma^-_0(\tau_0,\theta,\mu)
 &=
 \left(
 -\int_\theta^{s_{\rm m}}\frac{\mathsf G}{\mathsf F}\,\dd s,\,
 \tau_0+\int_\theta^{s_{\rm m}}\frac{\dd s}{\mathsf F}
 \right),\\
 \Gamma^+_0(s_{\rm out},\theta,\mu)
 &=
 \left(
 \int_{s_{\rm m}}^{s_{\rm out}}\frac{\mathsf G}{\mathsf F}\,\dd s,\,
 \tau_1(\theta,\mu)
 -\int_{s_{\rm m}}^{s_{\rm out}}\frac{\dd s}{\mathsf F}
 \right),
\end{aligned}
\tag{3.0e}
\]
where the integrands are evaluated at \((s,0,0,\mu)\).  If
\[
 s_{\rm out}=\mathscr S(\theta,\mu),\qquad \tau_0=0
\tag{3.0f}
\]
balances the first components and the two flight times, and
\((\mathsf G/\mathsf F)(\mathscr S(\theta,\mu),0,0,\mu)\) has a common
positive lower bound, then one common shrinking gives a unique joint
\(C^r\) zero of \(\mathfrak M_\eps\).

If the original phase equations are independent of the auxiliary clock,
the physical passage is obtained by projecting the matched orbit away from
\(\tau\).  A clock lift may then be translated to have incoming value zero;
the matching calculation does not assert that its auxiliary solution has
\(\tau_0=0\) for \(\eps>0\).
\end{lemma}

\begin{proof}
At the singular point \((3.0\mathrm f)\), put
\(\mathbf z_{{\rm m},0}=R\Gamma^-_0(0,\theta,\mu)\).  The second block of
\(\mathfrak M_0\) is exactly
\[
 \begin{pmatrix}
  \displaystyle\int_\theta^{s_{\rm out}}
        \frac{\mathsf G}{\mathsf F}\,\dd s\\[2mm]
  \displaystyle\tau_0+\int_\theta^{s_{\rm out}}
        \frac{\dd s}{\mathsf F}-\tau_1(\theta,\mu)
 \end{pmatrix}.
\tag{3.0g}
\]
In the unknown order \((\mathbf z_{\rm m},s_{\rm out},\tau_0)\),
\[
 D\mathfrak M_0=
 \begin{pmatrix}
  I_2&0&-\partial_{\tau_0}R\Gamma^-_0\\
  0&\multicolumn{2}{c}{
  \displaystyle
  \begin{pmatrix}
   \mathsf G/\mathsf F&0\\
   1/\mathsf F&1
  \end{pmatrix}_{s=\mathscr S(\theta,\mu)}
  }
 \end{pmatrix}.
\tag{3.0h}
\]
It is block triangular and its determinant is
\((\mathsf G/\mathsf F)(\mathscr S(\theta,\mu),0,0,\mu)>0\).
Compactness and the parameterized implicit-function theorem give the
claimed joint solution.

For \(\eps>0\),
\[
 \frac{\dd}{\dd t}(\xi+\eps\log n)
 =-\eps\mathsf G+\eps\mathsf G=0.
\]
Hence equality of the two \((\xi,\tau)\)-coordinates on the physical sheet
also forces equality of \(n\); at \(\eps=0\), both traces have \(n=0\).
Thus zeros of \(\mathfrak M_\eps\) are full augmented-orbit intersections,
not merely balance solutions.  Finally, translations
\(\tau\mapsto\tau+C\) leave every nonclock component unchanged, which
proves the clock-gauge statement.
\end{proof}

\begin{proof}[Proof of Theorem~\ref{thm:entryexit}]
Use the axial entry--exit coordinate \(s=-v\) and the normal coordinate
\(n=u\), and set
\[
\begin{aligned}
 \mathsf F(s,n,\eps,\mu)&=-g(n,-s,\eps,\mu),\\
 \mathsf G(s,n,\eps,\mu)&= f(n,-s,\eps,\mu),\\
 \mathsf H(s,n,\eps,\mu)&=-h_0(n,-s,\eps,\mu).
\end{aligned}
\tag{3.1}
\]
Then
\[
 \dot s=\eps\mathsf F+n\mathsf H,\qquad
 \dot n=n\mathsf G,
\tag{3.2}
\]
with
\[
 \mathsf F(s,0,0,\mu)>0,\qquad
 s\mathsf G(s,0,0,\mu)>0\quad(s\ne0).
\tag{3.3}
\]
If \(s_0=-\zeta<0\) and
\(s_1=-\mathcal B_\mu(\zeta)>0\), then
\[
 \int_{s_0}^{s_1}
 \frac{\mathsf G(s,0,0,\mu)}{\mathsf F(s,0,0,\mu)}\,\dd s=0.
\tag{3.4}
\]
These are Hsu's localized sign and balance hypotheses.  The two layer
trajectories required there are precisely \((2.5)\) in the new
coordinates.

It remains to prove joint dependence on \(\mu\).  Choose compact endpoint
and parameter sets \(S_-\Subset(-\infty,0)\) and \(K_\mu\Subset M\), and
shrink them once so that their product contains the base label and all
incoming and outgoing endpoint tubes below.  From now on this compact block
is fixed: every endpoint neighborhood, chart, spectral constant, transit
bound, and transversality constant is chosen once and is valid for every
\((\theta,\mu)\in S_-\times K_\mu\).  Write
\[
 s_1=\mathscr S(s_0,\mu):=-\mathcal B_\mu(-s_0).
\]
Normal rates at the two endpoint families are uniformly separated from
zero.  Introduce
\[
 \xi=-\eps\log n,\qquad \tau=\eps t,
\]
and, before any normal-form change, adjoin a frozen incoming endpoint
\(\theta\) and the coefficient parameter:
\[
\begin{aligned}
 \dot s&=\eps\mathsf F+n\mathsf H,&
 \dot n&=n\mathsf G,\\
 \dot\xi&=-\eps\mathsf G,&
 \dot\tau&=\eps,&
 \dot\theta&=0,&
 \dot\mu&=0.
\end{aligned}
\tag{3.5}
\]
At \(\eps=0\), \(n=0\) is a manifold of equilibria with center variables
\((s,\xi,\tau,\theta,\mu)\) and single normal eigenvalue
\(\mathsf G(s,0,0,\mu)\).

We first straighten the layer fibers on each endpoint tube.  Since
\(\mathsf G\) is bounded away from zero there, the noncharacteristic
transport problem
\[
 \mathsf H\,\partial_s\mathcal A+
 \mathsf G\,\partial_n\mathcal A=0,\qquad
 \mathcal A(s,0,\eps,\mu)=s
\tag{3.6}
\]
has a joint \(C^{r+4}\) solution.  The coordinate
\(\alpha=\mathcal A(s,n,\eps,\mu)\) satisfies
\[
 \dot\alpha=\eps\mathsf F\,\partial_s\mathcal A.
\tag{3.7}
\]
Thus, with \(n\) fast and
\(\widetilde{\mathbf z}=(\alpha,\xi,\tau,\theta,\mu)\) slow, each endpoint
system is
an exact standard slow--fast system
\[
 \dot n=n\,\widetilde{\mathsf G}(n,\widetilde{\mathbf z},\eps),
 \qquad
 \dot{\widetilde{\mathbf z}}
 =\eps\widetilde{\mathsf K}(n,\widetilde{\mathbf z},\eps).
\tag{3.8}
\]
The pushed-forward field is \(C^{r+3}\), because one derivative of
\(\mathcal A\) occurs in \((3.7)\).

Schecter's Sections~2.1, 2.4, and 2.5 and equations (2.9)--(2.16)
\cite{Schecter2008} prepare \((3.8)\), near each compact reduced endpoint
orbit, as follows, with Schecter's center variable \(c\) relabeled \(z\) to
avoid conflict with the return curvature:
\[
\begin{aligned}
 \dot x&=A(x,y,z,\eps)x,\\
 \dot y&=B(x,y,z,\eps)y,\\
 \dot z&=\eps\bigl(e_1+L(x,y,z,\eps)xy\bigr).
\end{aligned}
\tag{3.9}
\]
Here \(x\) corresponds to the sole incoming stable normal coordinate \(n\)
and \(y\) is absent; at the outgoing endpoint \(y\) corresponds to \(n\)
and \(x\) is absent, before the equivalent time reversal.  The prepared
center variable \(z\) is the image of
\((\alpha,\xi,\tau,\theta,\mu)\).  In the subsequently adapted split,
\(u_0\) is the reduced-flight coordinate, \(v_0\) contains the tangent
coordinates of the projected incoming family (including the frozen labels),
and the scalar \(w_0\) is the remaining balance coordinate.
The finite-regularity remark following Schecter's Theorem~2.1
(printed p.~413), together with the adapted construction in
Sections~2.3--2.4, gives a \(C^{r+1}\) normal form from a \(C^{r+3}\)
standard slow--fast field.  Lemma~\ref{lem:frozen-slices} shows that every
phase projection used in this preparation preserves the frozen
\((\theta,\mu)\)-slices.

The normal dimensions are \((\dim x,\dim y)=(1,0)\) at the incoming
endpoint and \((0,1)\) at the outgoing endpoint.  Hence \(xy\equiv0\) in
both applications, and the adapted center equations are
\[
 \dot u_0=\eps,\qquad \dot v_0=0,\qquad \dot w_0=0.
\tag{3.10}
\]
This verifies the rectified-speed hypothesis in the General Exchange
Lemma, with transit exponent one.  If \(d_\mu=\dim\mu\), the center and
incoming-manifold dimensions are
\[
 m_{\rm S}=4+d_\mu,\qquad p_{\rm S}=2+d_\mu,\qquad
 m_{\rm S}-p_{\rm S}-1=1.
\tag{3.11}
\]
The last dimension is the scalar balance coordinate.  Omitting the frozen
label \(\theta\) would make it zero.

We now specify the two graph families to which the exchange result is
applied.  Define the section footpoints and reduced travel time by
\[
\begin{aligned}
 \widehat\Lambda_-(\theta,\mu)&=-\Lambda_0(-\theta,\mu),&
 \widehat\Lambda_+(s_{\rm out},\mu)&=-\Lambda_1(-s_{\rm out},\mu),\\
 \tau_1(\theta,\mu)
 &=\int_\theta^{\mathscr S(\theta,\mu)}
     \frac{\dd s}{\mathsf F(s,0,0,\mu)}.
\end{aligned}
\]
Displaying the fixed normal coordinate \(n=u_1\), the incoming and
time-reversed ending manifolds before the final
prepared charts are the images
\[
\begin{aligned}
 \mathcal N_-^\eps(\tau_0,\theta,\mu)
 &=\Bigl(u_1,
    \mathcal A(\widehat\Lambda_-(\theta,\mu),\!u_1,\!\eps,\!\mu),
    -\eps\log u_1,\tau_0,\theta,\mu\Bigr),\\
 \mathcal N_+^\eps(s_{\rm out},\theta,\mu)
 &=\Bigl(u_1,
    \mathcal A(\widehat\Lambda_+(s_{\rm out},\mu),\!u_1,\!\eps,\!\mu),
    -\eps\log u_1,\tau_1(\theta,\mu),\theta,\mu\Bigr),
\end{aligned}
\tag{3.11a}
\]
with \(s_{\rm out}\) near \(\mathscr S(\theta,\mu)\).  These are joint
graph families over \((\tau_0,\theta,\mu)\) and
\((s_{\rm out},\theta,\mu)\), respectively.  Their
images remain \(C^{r+1}\) after preparation.  At \(\eps=0\), the
layer-footpoint property of \(\mathcal A\) makes their first center
coordinates \(\theta\) and \(s_{\rm out}\).  The time-reversed application changes
the reduced center direction to \(-\mathsf K_0\) but does not change the
point coordinates in \(\mathcal N_+^\eps\).

In the center-variable order
\((\alpha,\xi,\tau,\theta,\mu)\), the limiting incoming projection and the
time-reversed outgoing projection have the forms
\[
\begin{aligned}
 P_-^0(\tau_0,\theta,\mu)
   &=(\theta,0,\tau_0,\theta,\mu),\\
 P_+^0(s_{\rm out},\theta,\mu)
   &=(s_{\rm out},0,\tau_1(\theta,\mu),\theta,\mu).
\end{aligned}
\tag{3.11b}
\]
Both are immersions.  The forward reduced center direction is
\[
 \mathsf K_0=(\mathsf F,-\mathsf G,1,0_\theta,0_\mu).
\tag{3.11c}
\]
The outgoing application uses \(-\mathsf K_0\).

Here are Schecter's \textup{(E1)}--\textup{(E11)} with their actual objects;
this also fixes the dimension bookkeeping in both endpoint applications.
\begingroup
\ifsiadsreview\small\fi
\begin{enumerate}[label=\textup{(E\arabic*)},leftmargin=*,
                  itemsep=0pt,topsep=1pt,parsep=0pt]
\item The stable block is \(x=n\), \(k=1\); after time reversal at the
outgoing endpoint its rate is again at most \(-\nu_{\rm n}\).  The unstable
block is absent, \(y\in\mathbb R^0\), \(l=0\).  In the rectified prepared
center coordinates the center solution operator is the identity.
\item Take
\(\lambda_{\rm S}=-\nu_{\rm n}/2\),
\(\mu_{\rm S}=\nu_{\rm n}/4\), and choose
\(\beta_0>0\) so that \(r\beta_0<\nu_{\rm n}/4\) and
\(\max\{6,2r+1\}\beta_0<\nu_{\rm n}/4\).
\item The unions over \(\eps\) of the displayed
\(\mathcal N_-^\eps\) and \(\mathcal N_+^\eps\) are \(C^{r+1}\)
manifolds, uniformly over the compact frozen-label block.
\item This condition is automatic because \(l=0\), so \(y=0\) is the
entire prepared phase space in the absent unstable block.
\item The maps \(P_-^0\) and \(P_+^0\) in \((3.11\mathrm b)\) are
immersions; hence no nonzero tangent vector is annihilated by the center
projection.
\item In the adapted starting coordinates, \(\dot w_0=0\) exactly.
\item Since \(\dot u_0=\eps\), the transit exponent is \(a=1\), and the
constant \(K_3\) may be chosen below \(1\).
\item The adapted starting chart and the transformed vector field are joint
\(C^{r+1}\).
\item In the ending chart,
\(v^1\in\mathbb R^{p_{\rm S}+1}=\mathbb R^{3+d_\mu}\) and
\(w^1\in\mathbb R\); the flowed center projection is \(w^1=0\).
\item The ending chart and its transformed field are joint \(C^{r+1}\);
at the repelling endpoint these are constructed for the reversed system.
\item On \(w^1=0\), \(\dot u_0=\eps\) gives transit time
\(u^1/\eps\).  Choose
\(\beta_0<\beta<\beta_1\) so that
\(\mu_{\rm S}-\max\{6,2r+1\}\beta_1>0\), and shrink the ending
\(u^1\)-window \([K_1,K_2]\) so that
\(1<K_2/K_1<\beta_1/\beta\), uniformly in \((\theta,\mu)\).
\end{enumerate}
At the incoming projection \(\mathsf K_0\) has positive
\(\xi\)-component; at the outgoing projection the reversed vector
\(-\mathsf K_0\) does as well.  This reduced-flow transversality constructs
the relative flow boxes used in \textup{(E6)}--\textup{(E9)}.  It is
separate from \textup{(E5)}, which is the injectivity of center projection.
Thus all hypotheses of Corollary~\ref{cor:augmented-exchange} hold for the
two displayed families with one uniform choice of neighborhoods and
constants.
\endgroup

We now construct the full middle intersection.  Choose the fixed section
\(\Sigma_{\rm m}=\{s=s_{\rm m}\}\), with \(s_{\rm m}=0\), after shrinking
the compact label block so that every incoming endpoint lies to its left
and every balanced outgoing endpoint lies to its right.  Between fixed
small endpoint neighborhoods, the limiting exponential coordinate satisfies
\(\xi\ge2\nu>0\).  The exchange estimates are uniform, so the perturbed
middle pieces satisfy \(\xi\ge\nu\).

Apply Lemma~\ref{lem:exchange-to-trace} first to the forward exchanged
family issued from \(\mathcal N_-^\eps\), with free endpoint variable
\(\tau_0\), and then to the reversed outgoing family issued from
\(\mathcal N_+^\eps\), with free endpoint variable \(s_{\rm out}\).
The ending sections are chosen inside the two compact endpoint tubes.
The reduced vectors \(\mathsf K_0\) and \(-\mathsf K_0\) have uniformly
positive \(\xi\)-components there by \((3.11\mathrm c)\), so the required
endpoint intersections are transverse.  Shrinking once more, the regular
segments from those sections to \(\Sigma_{\rm m}\) stay in the strip
\(\xi\ge\nu\).  Since \(A_0=\mathsf F>0\), the \(s\)-coordinate is
strictly monotone in the incoming orientation and in the opposite direction
for the reversed outgoing field.  Hence each family meets every relevant
intermediate section \(s=\text{constant}\), in particular
\(\Sigma_{\rm m}\), exactly once.

The endpoint families in \((3.11\mathrm a)\) lie on the invariant sheet
\((3.0\mathrm c)\).  In slow time \(\sigma=\eps t\), its physical
middle equations are
\[
\begin{aligned}
 \frac{\dd s}{\dd\sigma}
 &=\mathsf F(s,N,\eps,\mu)
   +\eps^{-1}N\mathsf H(s,N,\eps,\mu)=:A_\eps,\\
 \frac{\dd\xi}{\dd\sigma}
 &=-\mathsf G(s,N,\eps,\mu),\qquad
 \frac{\dd\tau}{\dd\sigma}=1.
\end{aligned}
\tag{3.12}
\]
Lemma~\ref{lem:exp-flatness} gives a joint \(C^r\) extension on
\(\xi\ge\nu\), with
\(A_0=\mathsf F(s,0,0,\mu)>0\).  We may therefore use \(s\) as independent
variable:
\[
 \frac{\dd\xi}{\dd s}=-\frac{\mathsf G}{A_\eps},
 \qquad
 \frac{\dd\tau}{\dd s}=\frac1{A_\eps}.
\tag{3.13}
\]
Consequently Lemma~\ref{lem:exchange-to-trace} transports the two exchanged
families to \(\Sigma_{\rm m}\) as the joint \(C^r\) traces
\[
 \Gamma^-_\eps=(\Xi^-_\eps,T^-_\eps),\qquad
 \Gamma^+_\eps=(\Xi^+_\eps,T^+_\eps),
\tag{3.14}
\]
parameterized respectively by \((\tau_0,\theta,\mu)\) and
\((s_{\rm out},\theta,\mu)\).

At \(\eps=0\), integration of \((3.13)\) gives, as identities on fixed
neighborhoods,
\[
\begin{aligned}
 \Gamma^-_0(\tau_0,\theta,\mu)
 &=
 \left(
 -\int_\theta^{s_{\rm m}}\frac{\mathsf G}{\mathsf F}\,\dd s,\,
 \tau_0+\int_\theta^{s_{\rm m}}\frac{\dd s}{\mathsf F}
 \right),\\
 \Gamma^+_0(s_{\rm out},\theta,\mu)
 &=
 \left(
 \int_{s_{\rm m}}^{s_{\rm out}}\frac{\mathsf G}{\mathsf F}\,\dd s,\,
 \tau_1(\theta,\mu)
 -\int_{s_{\rm m}}^{s_{\rm out}}\frac{\dd s}{\mathsf F}
 \right),
\end{aligned}
\tag{3.15}
\]
where the integrands are evaluated at \((s,0,0,\mu)\).  Thus the second
block of the explicit map \(\mathfrak M_0\) in \((3.0\mathrm d)\) is
\[
 \begin{pmatrix}
  \displaystyle\int_\theta^{s_{\rm out}}
        \frac{\mathsf G}{\mathsf F}\,\dd s\\[2mm]
  \displaystyle\tau_0+\int_\theta^{s_{\rm out}}
        \frac{\dd s}{\mathsf F}-\tau_1(\theta,\mu)
 \end{pmatrix}.
\tag{3.16}
\]
Equation~\((3.4)\) and the definition of \(\tau_1\) show that
\[
 s_{\rm out}=\mathscr S(\theta,\mu),\qquad \tau_0=0
\]
is its singular zero.  Its Jacobian in
\((s_{\rm out},\tau_0)\) equals
\[
 \begin{pmatrix}
  \mathsf G/\mathsf F&0\\
  1/\mathsf F&1
 \end{pmatrix}_{s=\mathscr S(\theta,\mu)},
\qquad
 \det=\left.\frac{\mathsf G}{\mathsf F}
 \right|_{s=\mathscr S(\theta,\mu)}>0,
\tag{3.17}
\]
with a common positive lower bound.  Lemma~\ref{lem:middle-clock} now gives
the unique full augmented intersection jointly in
\((\theta,\eps,\mu)\).  In particular, there is no suppressed normal
equation: equality of \(\xi\) on the invariant physical sheet forces
equality of \(n\).

Finally restrict the frozen endpoint label to
\[
 \theta=-\pi_0(v,\mu)
\]
and project the matched orbit away from the auxiliary variables to the
physical phase variables.  Since \(\tau\) is a decoupled clock,
Lemma~\ref{lem:middle-clock} permits the projected orbit to be represented,
after this projection, by the gauge choice \(\tau_0=0\).  At \(\eps=0\),
the incoming layer, balanced critical segment, and outgoing layer give respectively
\(\pi_0\), \(\mathcal B_\mu\), and \(\pi_1^{-1}\).  This proves the joint
extension and formula \((2.6)\).
\end{proof}

\begin{remark}
Theorem~\ref{thm:entryexit} does not cover a reduced slow equilibrium, a
higher-order normal factor, loss of normal hyperbolicity at an endpoint, or
a general critical curve before an exact reduction to \((2.1)\).
Degenerate turning-point problems in which the reduced slow flow itself has
an equilibrium require a different blow-up and entry--exit relation; see
\cite{HuzakKristiansen2025}.
For the stricter subclass \(h_0\equiv0\), parameter-dependent
\(C^\infty\) passage smoothness also follows directly from
Corollary~1.2 and Remark~1 of De Maesschalck--Schecter
\cite{DeMaesschalckSchecter2016}.
\end{remark}

\section{Continuous quadratic grazing}

This section supplies the second local module.  The \(3/2\) exponent is the
classical continuous-grazing scale; the proposition records the precise
moving-coordinate coefficient and the joint derivative control required by
the later composition.

\siadsneedspace{.42\textheight}
\begin{proposition}[Continuous-grazing \(3/2\) transition]
\label{prop:grazing}
Under \((2.7)\)--\((2.10)\), let \(\mathcal D(\rho,p)\) be the outgoing
\(I_p\)-coordinate of the physical piecewise-smooth excursion minus that
of the reference \(X_p^-\)-excursion, for the same incoming point.  Then
\[
 \mathcal D(\rho,p)=
 \begin{cases}
 0,&\rho\le0,\\
 \rho^{3/2}\mathcal A_{\rm gr}(\sqrt\rho,p),&\rho\ge0,
 \end{cases}
\tag{4.1}
\]
where \(\mathcal A_{\rm gr}(r,p)\) is smooth for \(r\ge0\), jointly with
\(p\), and
\[
 \mathcal A_{\rm gr}(0,p)
 =\frac{4\sqrt2}{3\sqrt{\kappa(p)}}\,\beta(p).
\tag{4.2}
\]
Equivalently, put \(C(p):=\mathcal A_{\rm gr}(0,p)\).  Then, for a smooth
\(B\),
\[
 \mathcal D(\rho,p)
 =C(p)\rho_+^{3/2}+\rho_+^2B(\sqrt{\rho_+},p)
\tag{4.3}
\]
with all prescribed finite ramified derivatives bounded uniformly.
\end{proposition}

\begin{proof}
A parameter-dependent \(X_p^-\)-flow box gives coordinates \((t,I)\) in
which \(X_p^-=\partial_t\), the grazing orbit is \(I=0\), and
\[
 h_p(t,I)=I-\frac{\kappa(p)}2t^2
 +O(|t|^3+|tI|+I^2).
\tag{4.4}
\]
Hadamard factorization is joint in \(p\), and in these coordinates
\[
 \dot t=1+h_pa_p(t,I),\qquad
 \dot I=h_pb_p(t,I),\qquad b_p(0,0)=\beta(p).
\tag{4.5}
\]
For a penetrating orbit put \(I_{\rm in}=r^2\) and \(t=ru\), \(r\ge0\).
Taylor division gives a smooth function
\[
 \widehat h(r,u,w,p)
 :=r^{-2}h_p(ru,r^2+r^3w)
 =1-\frac{\kappa(p)}2u^2+r\widehat H(r,u,w,p).
\tag{4.5a}
\]
For \(w=0\), the switching equation has a smooth incoming root
\(u_-(r,p)\), whose limiting value, together with the prospective outgoing
value, is
\[
 u_\pm(0,p)=\pm\sqrt{2/\kappa(p)}.
\]
On the plus excursion write \(I=r^2+r^3w(u)\).  Dividing
\(\dd I/\dd t=h_pb_p/(1+h_pa_p)\) by \(r^2\) yields the explicit scaled
equation
\[
 \frac{\dd w}{\dd u}
 =\frac{\widehat h(r,u,w,p)\,
 b_p(ru,r^2+r^3w)}
 {1+r^2\widehat h(r,u,w,p)\,
 a_p(ru,r^2+r^3w)},
 \qquad w(u_-(r,p))=0.
\tag{4.5b}
\]
Its right-hand side is smooth up to \(r=0\), uniformly on a common
\(u\)-interval.  At \(r=0\), the switching function along this solution is
\(1-\kappa(p)u^2/2\), whose positive zero is simple.  The smooth ODE
theorem followed by the implicit-function theorem therefore gives a smooth
outgoing root \(\widetilde u_+(r,p)\) of
\(
 \widehat h(r,u,w(u;r,p),p)=0
\), jointly in \((r,p)\).  Hence
\[
 I_{\rm out}-I_{\rm in}
 =r^3w(\widetilde u_+(r,p);r,p)
 =r^3\mathcal A_{\rm gr}(r,p),
\]
with \(\mathcal A_{\rm gr}\) smooth.
At \(r=0\),
\[
\begin{aligned}
 \mathcal A_{\rm gr}(0,p)
 &=\beta(p)
 \int_{-\sqrt{2/\kappa(p)}}^{\sqrt{2/\kappa(p)}}
 \left(1-\frac{\kappa(p)}2u^2\right)\,\dd u\\
 &=\frac{4\sqrt2}{3\sqrt{\kappa(p)}}\,\beta(p).
\end{aligned}
\]
The reference orbit does not enter the plus side for \(\rho\le0\), so its
correction is zero there.  This proves \((4.1)\)--\((4.3)\).
\end{proof}

\begin{remark}[Jet formula]
\label{rem:jet}
In coordinates \(h=y\), write \(X^\pm=(f^\pm,g^\pm)\), and suppose at the
grazing point
\[
 f^\pm(0,0)=v\ne0,\qquad g^\pm(0,0)=0.
\]
Then
\[
 \kappa=-v\,\partial_xg^-(0,0),\qquad
 \beta=\partial_yg^+(0,0)-\partial_yg^-(0,0).
\tag{4.6}
\]
For smooth one-sided branches, trace continuity makes \(3/2\) the first
possible generic nonsmooth order.  Its coefficient is nonzero precisely
when \(\beta\ne0\); continuity alone does not force this condition.  A jump
only in the normal derivative of the tangential velocity can therefore
cancel at this order.
\end{remark}

\section{Moving-threshold return composition}

The inputs are the physical passage from Section~3, the local grazing
transition from Section~4, and the regular tubes in \textup{(REG)}.  The
output is the structural return germ used by every subsequent zero-counting
argument.  The coordinate must move with the grazing graph; the two remarks
after the proof explain why a fixed limiting coordinate would not preserve
the required derivative information.

\begin{lemma}[Moving-threshold return composition]
\label{lem:composition}
Under \textup{(EE)}, \textup{(GR)}, \textup{(REG)}, and \textup{(BAL)},
the displacement has the form
\((2.15)\), and its leading coefficient is \((2.16)\).
\end{lemma}

\begin{proof}
Invoke Theorem~\ref{thm:entryexit} with \(r=\ell+1\).  All remaining
reference transitions are regular, so the pre-grazing map, \(Q_p\),
\(Q_p^{-1}\), the post-grazing map, and the reference return are jointly
\(C^{\ell+1}\).  In particular \(S\) is jointly \(C^{\ell+1}\).

Parameterize the outgoing grazing section by the normalized \(I_p\)
coordinate \(u\).  Let \(u_0(q,p)\) be the reference outgoing coordinate
and let \(R_p\) be the composition of all maps from that section to the
final \(q\)-coordinate.  Proposition~\ref{prop:grazing} supplies
\[
 \mathcal D(q,p)=q_+^{3/2}\mathcal C(\sqrt{q_+},p),
\qquad
 \mathcal C(0,p)=
 \frac{4\sqrt2}{3\sqrt{\kappa(p)}}\,\beta(p).
\tag{5.1}
\]
Regard \(s\ge0\) and \(q\) as independent in the coefficient chart.  On
the physical penetrating side \(q=s^2\), and
\[
\begin{aligned}
 &R_p\!\left(u_0(q,p)+s^3\mathcal C(s,p)\right)-R_p(u_0(q,p))\\
 &\quad=s^3\mathcal C(s,p)
 \int_0^1
 \partial_uR_p\!\left(u_0(q,p)+t\,s^3\mathcal C(s,p)\right)\,\dd t .
\end{aligned}
\tag{5.2}
\]
Hence the coefficient in \((2.15)\) is explicitly
\[
 K(s,q,p)=\mathcal C(s,p)\int_0^1
 \partial_uR_p\!\left(
 u_0(q,p)+t\,s^3\mathcal C(s,p)
 \right)\,\dd t .
\tag{5.2a}
\]
It is \(C^\ell\) because \(R_p\) is \(C^{\ell+1}\).
This proves \((2.15)\).

Let \(T_0\) be the post-grazing transition to the original \(\eta\)
coordinate.  Reference closure gives \(T_0(0)=0\).  If
\(X_{{\rm reg},p}\) denotes the smooth vector field along this regular
tube, the scalar section-map formula is
\[
 T_p'(u)=J_{\rm sec}(u,p)
 \exp\!\left(
 \int_0^{\tau(u,p)}
 \operatorname{div}X_{{\rm reg},p}(\varphi_p^t(u))\,\dd t
 \right),
\tag{5.3}
\]
where the transverse flux factor \(J_{\rm sec}\) is nonzero.  Bounded
flight time and uniform transversality show that \(T_0'(0)\ne0\).
Consequently
\[
 L_0
 =\left.\partial_u[Q_0(T_0(u))]\right|_{u=0}
 =\partial_\eta Q_0(0)\,T_0'(0)\ne0.
\tag{5.4}
\]
Evaluating \((5.2)\) at the origin gives \((2.16)\).
\end{proof}

\begin{remark}[Why the coordinate moves]
\label{rem:moving}
The implicit-function theorem gives a unique grazing graph
\(\eta=\eta_{\rm g}(p)\) satisfying
\(\rho(\eta_{\rm g}(p),p)=0\).  If, for example,
\(\rho(\eta,\eps)=\eta-\gamma\eps+o(\eps)\), then
\[
 (\eta-\gamma\eps)_+^{3/2}-\eta_+^{3/2}
\]
is of leading order when \(\eta\asymp\eps\).  Formula \((2.15)\) is
uniform through such regimes precisely because \(q=\rho(\eta,p)\) is the
exact moving penetration.
\end{remark}

\begin{remark}[Conjugacy versus pullback]
\label{rem:pullback}
If one merely pulls the old displacement back by \(Q_p^{-1}\), obtaining
\(\Delta^{\rm pb}\), then the actual conjugated displacement is
\[
 \Delta(q,p)=M(q,p)\Delta^{\rm pb}(q,p),
\quad
 M(q,p)=\int_0^1\partial_\eta Q_p\!
 \left(Q_p^{-1}(q)+t\Delta^{\rm pb}(q,p)\right)\,\dd t.
\tag{5.5}
\]
Thus the output-coordinate derivative is part of \(L_0\); omitting it can
change the reported coefficient.
\end{remark}

\section{Sharp zero count and rank-two unfolding}

This section turns the return germ from Section~5 into the cyclicity theorem.
Throughout, \(\ell\ge2\) is the fixed integer in
Theorem~\ref{thm:main}, and \(S,K\) denote the corresponding
representatives.
Its decisive input is derivative-level control on each side of the moving
threshold, rather than a value expansion alone.  The first lemma supplies
the two one-sided monotonicities, the next proposition counts zeros, and the
last proposition proves that the applicable bound is attained in each sign
class.

Subtract the constant and linear smooth jets:
\[
 S_2(q,p)=S(q,p)-a(p)-b(p)q.
\tag{6.1}
\]
Then
\[
 \Delta(q,p)
 =a(p)+b(p)q+S_2(q,p)
 +q_+^{3/2}K(\sqrt{q_+},q,p),
\tag{6.2}
\]
where \(S_2(0,p)=S_{2,q}(0,p)=0\).

\begin{lemma}[One-sided monotonicity of the return derivative]
\label{lem:curvature}
If \(cd\ne0\), then, uniformly as \((q,p)\to(0,0)\) with \(q<0\),
\[
 \Delta_{qq}(q,p)=S_{qq}(q,p)\longrightarrow2c.
\tag{6.3}
\]
For \(q=s^2>0\), define the ramified restriction
\[
 H(s,p)=K(s,s^2,p),
\]
and interpret \(H_s,H_{ss}\) as total \(s\)-derivatives.  Then
\[
\begin{aligned}
 J(s,p)
 &:=\partial_s[\Delta_q(s^2,p)]\\
 &=2sS_{qq}(s^2,p)
 +\frac32H(s,p)+\frac52sH_s(s,p)
 +\frac12s^2H_{ss}(s,p)
 \longrightarrow\frac32d
\end{aligned}
\tag{6.4}
\]
uniformly as \((s,p)\to(0,0)\).
\end{lemma}

\begin{proof}
Equation \((6.3)\) is ordinary \(C^2\) continuity of \(S\).
For \(s>0\), the fractional term restricted to the ramified chart is
\(s^3H(s,p)\).  Since \(\partial_q=(2s)^{-1}\partial_s\),
differentiate once to obtain \((6.4)\).  Joint \(C^2\) regularity in the
ramified variables gives the stated limit.
\end{proof}

\begin{proposition}[Sign-sharp local zero count]
\label{prop:zeros}
Let \(\Delta\) satisfy \((6.2)\) and the estimates in
Lemma~\ref{lem:curvature}.  There are a fixed \(\delta>0\) and a parameter
neighborhood such that
\[
 \#\{q\in(-\delta,\delta):\Delta(q,p)=0\}
 \le
 \begin{cases}
 2,&cd>0,\\
 3,&cd<0.
 \end{cases}
\tag{6.5}
\]
No unfolding-rank assumption is needed for this upper bound.
\end{proposition}

\begin{proof}
After shrinking, \(\Delta_q\) is strictly monotone on the negative
half-interval by \((6.3)\).  By \((6.4)\), it is strictly monotone as a
function of \(s=\sqrt q\) on the positive half-interval.  If \(c\) and
\(d\) have the same sign, these monotonicities join across the continuous
value \(\Delta_q(0,p)\), so \(\Delta_q\) has at most one zero.  Rolle's
theorem gives at most two zeros of \(\Delta\).

If the signs are opposite and \(\Delta_q(0,p)\ne0\), then
\(\Delta_q\) has at most one zero on each open half-interval and none at
the boundary, hence at most two in total.  If
\(\Delta_q(0,p)=0\), the opposed strict monotonicities force
\(\Delta_q\) to have the same strict sign on both open half-intervals, so
the boundary zero is its only zero.  In either case, four distinct zeros
of the \(C^1\) function \(\Delta\) would force three distinct zeros of its
derivative.  Thus \(\Delta\) has at most three zeros.
\end{proof}

\begin{proposition}[Sharpness under rank two]
\label{prop:sharpness}
If \((2.14)\) holds, there is a \(C^\ell\) curve \(\mu_*(\eps)\) satisfying
\((2.18)\).  For every sufficiently small fixed \(\eps>0\), the applicable
bound in \((6.5)\) is attained by simple zeros along parameter paths
converging to \(\mu_*(\eps)\).
\end{proposition}

\begin{proof}
Put
\[
 \Phi(\eps,\mu)=(a(\eps,\mu),b(\eps,\mu)).
\tag{6.6}
\]
Because \(S\in C^{\ell+1}\), one has
\(a=S(0,\cdot)\in C^{\ell+1}\) and
\(b=S_q(0,\cdot)\in C^\ell\), hence \(\Phi\in C^\ell\).
The parameter-dependent inverse-function theorem gives a \(C^\ell\)
inverse chart
\[
 \mu=\Psi(\eps,\mathsf a,\mathsf b),\qquad
 \Phi\bigl(\eps,\Psi(\eps,\mathsf a,\mathsf b)\bigr)
   =(\mathsf a,\mathsf b),\qquad
 \Psi(0,0,0)=0,
\]
on fixed neighborhoods for all sufficiently small \(\eps\).  Define
\[
 \mu_*(\eps)=\Psi(\eps,0,0),\qquad
 p_*(\eps)=(\eps,\mu_*(\eps)).
\]
Then \(a(p_*(\eps))=b(p_*(\eps))=0\), and \(C^1\) regularity gives
\(\mu_*(\eps)=O(\eps)\).  Thus, for fixed \(\eps\),
\(\Psi(\eps,\mathsf a,\mathsf b)\to\mu_*(\eps)\) as
\((\mathsf a,\mathsf b)\to(0,0)\), whereas jointly
\((\eps,\mu_*(\eps))\to(0,0)\).

Fix a compact parameter neighborhood \(\mathcal P_1\) of the origin in the
common coordinate domain, and shrink the target of \(\Psi\) so that its
image lies in \(\mathcal P_1\).  Taylor's formula and the ramified expansion
give
\[
\begin{aligned}
 S_2(q,p)&=c_pq^2+q^2\omega_-(q,p),
       &&c_p=\tfrac12S_{qq}(0,p),\quad\omega_-\to0,\\
 q^{3/2}K(\sqrt q,q,p)
 &=q^{3/2}\bigl(d_p+\omega_+(q,p)\bigr),
       &&d_p=K(0,0,p),\quad\omega_+\to0 .
\end{aligned}
\tag{6.7}
\]
Here the first line is used for \(q<0\), whereas in the second line
\(q>0\) and
\(\omega_+(q,p)=K(\sqrt q,q,p)-K(0,0,p)\).  The remainder statements mean
\[
 \sup_{p\in\mathcal P_1}|\omega_-(q,p)|\longrightarrow0
 \quad(q\uparrow0),
 \qquad
 \sup_{p\in\mathcal P_1}|\omega_+(s^2,p)|\longrightarrow0
 \quad(s\downarrow0).
\]
The latter is a one-sided limit in the ramified variable; no additional
ordinary \(q\)-regularity is asserted.  Hence the estimates remain uniform
when \(q\to0\) at fixed \(\eps\) and \(p\to p_*(\eps)\); they do not require
\(p\to0\) for that fixed \(\eps\).
Here \(c_p\to c\) and \(d_p\to d\) as \(p\to0\).  For a fixed
\(\eps\), set \(c_\eps=c_{p_*(\eps)}\) and
\(d_\eps=d_{p_*(\eps)}\).  Along the inverse chart,
\[
 c_p\to c_\eps,\qquad d_p\to d_\eps
 \quad\text{as }(\mathsf a,\mathsf b)\to(0,0).
\]
Moreover \(c_\eps\to c\) and \(d_\eps\to d\) as \(\eps\downarrow0\), so
their signs agree with those of \(c,d\) for all sufficiently small
\(\eps\).
The corresponding first-derivative estimates follow from
Lemma~\ref{lem:curvature}.

For this fixed \(\eps\), put \(\sigma=\sgn(c_\eps)\).  Replacing
\(\Delta\) by \(\sigma\Delta\), the target coordinates by
\((\widetilde{\mathsf a},\widetilde{\mathsf b})
 =\sigma(\mathsf a,\mathsf b)\), and the inverse chart by
\[
 \widetilde\Psi(\eps,\widetilde{\mathsf a},
 \widetilde{\mathsf b})
 =\Psi(\eps,\sigma\widetilde{\mathsf a},
 \sigma\widetilde{\mathsf b})
\]
does not change the zeros or the center \(\mu_*(\eps)\).  Suppressing
tildes, we may therefore assume \(c_\eps>0\).

If \(cd>0\), then \(d_\eps>0\).  In the inverse chart choose
\((\mathsf a,\mathsf b)=(-A,0)\), \(A>0\).  On the two scales
\[
 q=-A^{1/2}z,\qquad q=A^{2/3}z,
\]
division by \(A\) gives respectively
\[
 -1+c_pz^2+o(1),\qquad -1+d_pz^{3/2}+o(1).
\tag{6.8}
\]
As \(A\downarrow0\), the limiting coefficients are
\(c_\eps,d_\eps>0\).  More explicitly, with
\(p_A=(\eps,\Psi(\eps,-A,0))\), define on compact positive
\(z\)-intervals
\[
 F_A^-(z)=A^{-1}\Delta(-A^{1/2}z,p_A),
 \qquad
 F_A^+(z)=A^{-1}\Delta(A^{2/3}z,p_A).
\]
The uniform remainder estimates in \((6.7)\), together with
Lemma~\ref{lem:curvature} after the corresponding rescalings, give
\[
 F_A^-\longrightarrow-1+c_\eps z^2,
 \qquad
 F_A^+\longrightarrow-1+d_\eps z^{3/2}
 \quad\text{in }C^1
\tag{6.8a}
\]
on compact neighborhoods of
\(c_\eps^{-1/2}\) and \(d_\eps^{-2/3}\), respectively.  Both limiting
roots are simple.  Simple-root persistence therefore gives exactly one
simple root in each of these two disjoint scaled neighborhoods, while
\(\mu=\Psi(\eps,-A,0)\to\mu_*(\eps)\).

If \(cd<0\), then \(d_\eps<0\), and choose
\[
 \mathsf b=t,\qquad
 \mathsf a=\frac{t^2}{8c_\eps},\qquad t>0.
\tag{6.9}
\]
Take
\(p=(\eps,\Psi(\eps,\mathsf a,\mathsf b))\).
On \(q=tz<0\), division by \(t^2\) gives
\[
 \frac1{8c_\eps}+z+c_pz^2+o(1).
\tag{6.10}
\]
On compact neighborhoods of its two negative roots,
\(t^{-2}\Delta(tz,p_t)\) converges to this quadratic in \(C^1\), by
\((6.7)\) and Lemma~\ref{lem:curvature}.  Its discriminant is \(1/2\), so
both roots are simple and persist as two simple negative roots for small
\(t>0\).  On the positive side set
\[
 q_t=\left(\frac{t^2}{4c_\eps|d_\eps|}\right)^{2/3}.
\tag{6.11}
\]
Choose once and for all \(0<\lambda<2^{-2/3}\).  Since
\(\mathsf a=t^2/(8c_\eps)\), the uniform expansion gives
\[
 \frac{\Delta(\lambda q_t,p_t)}{\mathsf a}
  =1-2\lambda^{3/2}+o(1)>0,\qquad
 \frac{\Delta(q_t,p_t)}{\mathsf a}
  =-1+o(1)<0,
\tag{6.12}
\]
where \(p_t=(\eps,\Psi(\eps,\mathsf a,t))\).  Hence there is a root
\[
 q_*(t)\in(\lambda q_t,q_t),\qquad q_*(t)\asymp t^{4/3}.
\tag{6.13}
\]
It is simple.  Indeed, with \(q=s^2\), Lemma~\ref{lem:curvature} gives
\[
 \Delta_q(s^2,p_t)
 =t+\int_0^sJ(\sigma,p_t)\,\dd\sigma
 =t+s\left(\frac32d_\eps+o(1)\right).
\tag{6.14}
\]
On the bracket, \(s\ge\sqrt{\lambda q_t}\asymp t^{2/3}\).
Because \(d_\eps<0\), the negative \(t^{2/3}\) term dominates the positive
\(t\) term, so \(\Delta_q<0\) throughout the bracket for small \(t\).
Finally
\[
 \mu=\Psi\left(\eps,\frac{t^2}{8c_\eps},t\right)
 \longrightarrow\mu_*(\eps),
\]
which proves the asserted fixed-\(\eps\) centering.
\end{proof}

\begin{proof}[Proof of Theorem~\ref{thm:main}]
The ramified expansion and coefficient formula are
Lemma~\ref{lem:composition}.  Lemma~\ref{lem:curvature} and
Proposition~\ref{prop:zeros} give the upper bounds, while
Proposition~\ref{prop:sharpness} gives the centered scalar sharpness and
the curve \((2.18)\).  For \(\eps>0\),
\(\mathcal R_p=P_p^{\rm phys}\) is an ordinary Poincar\'e map on the
selected local first-return section.  Its distinct fixed points correspond
bijectively to distinct periodic orbits.  The \(\eps=0\) member is used only
as the singular parameter extension in the displacement estimates, so no
cycle correspondence is invoked there.  To obtain equality in the physical
cyclicity statement, choose \(\eps_n\downarrow0\) and, on the realizing path
at \(\eps_n\), choose \(\mu_n\) and \(N\) distinct zeros \(q_{n,j}\) of
\(\Delta(\,\cdot\,,(\eps_n,\mu_n))\) so that
\[
 \|\mu_n-\mu_*(\eps_n)\|
 +\max_{1\le j\le N}|q_{n,j}|<\frac1n,
\]
where \(N=2\) or \(3\) in the indicated sign case.  Then
\((\eps_n,\mu_n)\to(0,0)\), every \(\eps_n>0\), and
Definition~\ref{def:cyclicity} gives equality.  This completes the proof.
\end{proof}

\begin{remark}[A \(C^1\) value expansion is insufficient]
\label{rem:c1}
The curvature hypotheses cannot be replaced by a pointwise remainder.
For \(x<0\), set
\[
 r(x)=(-x)^{7/2}\sin((-x)^{-2}),\qquad r(x)=0\quad(x\ge0).
\]
Then \(r\in C^1\), \(r=o(x^2)\), and its \(C^1\) norm tends to zero on
shrinking intervals, but
\[
 a+bx+x^2+x_+^{3/2}+r(x)
\]
can have arbitrarily many negative zeros tending to the origin.  The
oscillatory phase changes by an unbounded amount on a window where the
nonoscillatory quadratic drift is smaller than the amplitude.  This is why
\((6.3)\)--\((6.4)\), rather than only the displayed leading values, are
part of the theorem.
\end{remark}

\begin{remark}[Germ quantifiers]
The interval \((-\delta,\delta)\) is fixed before the parameter
neighborhood is shrunk.  The opposite-sign exact polynomial can have a
fourth zero outside such a sufficiently small local interval; no global
zero bound is asserted.
\end{remark}

\begin{remark}[Coordinate meaning]
\label{rem:coordinates}
The coefficient \(d\) depends on oriented input and output normalizations,
but its nonvanishing does not.  Under an orientation-preserving phase
change \(\widetilde q=\alpha q+O(q^2)\), \(\alpha>0\), and a nonzero
output scaling, \(c\) and \(d\) acquire the same output sign and positive
phase factors.  Hence \(\sgn(cd)\), and therefore the distinction in
\((2.17)\), is invariant.
\end{remark}

\section{An explicit sharpness benchmark}

The following polynomial family verifies all general hypotheses and
independently controls the signs of \(c\) and \(d\).  It proves nonemptiness
of the abstract class and sharpness of both alternatives without computer
assistance; it is a mathematical benchmark, not a biological model.

Let
\[
\begin{aligned}
 \phi(x)&=x(1-x^2)\left(x^2-\frac37\right),\\
 \psi_0(x)&=x,\qquad
 \psi_1(x)=x^3-\frac35x,\\
 m_\mu(x)&=1+\alpha\phi(x)+\mu_0\psi_0(x)+\mu_1\psi_1(x).
\end{aligned}
\tag{7.1}
\]
Fix nonzero \(\alpha,\beta\), with \(|\alpha|\) small enough that
\(m_0>0\) near \([-1,1]\), and restrict \(\mu\) so that \(m_\mu>0\)
there.  With switching function \(h(x,y)=y-1\), define
\[
 X^-_{\eps,\mu}(x,y)=
 \begin{pmatrix}
 \eps-m_\mu(x)y^2\\
 m_\mu(x)xy
 \end{pmatrix},
\qquad
 X^+_{\eps,\mu}
 =X^-_{\eps,\mu}+\beta(y-1)\begin{pmatrix}0\\1\end{pmatrix}.
\tag{7.2}
\]

\begin{proposition}[Explicit realization of both sign cases]
\label{prop:benchmark}
The family \((7.2)\), with suitable separated sections, satisfies all
hypotheses of Theorem~\ref{thm:main}.  In the normalized moving-penetration
coordinate its coefficients are
\[
 c=\frac{8\alpha}{7},\qquad
 d=\frac{4\sqrt2}{3}\beta,\qquad
 L_0=1,
\tag{7.3}
\]
and
\[
 \det D_\mu(a,b)(0)=\frac8{15}.
\tag{7.4}
\]
Thus the signs of \(c\) and \(d\) can be chosen independently.
\end{proposition}

\begin{proof}
At \(\eps=0\), \(y=0\) is a critical line with normal eigenvalue
\(x\,m_\mu(x)\), and the reduced slow flow is \(x'=1\).  Near the entry
endpoint \(x=-1\), define \(p_\mu(x)>0\) by
\[
 \int_x^{p_\mu(x)} f_\mu(s)\,\dd s=0,\qquad
 f_\mu(s)=s\,m_\mu(s).
\tag{7.5}
\]
The choice of \(\phi\) gives
\[
\begin{aligned}
 \int_{-1}^{1}f_0(s)\,\dd s
 &=\alpha\int_{-1}^{1}
 s^2(1-s^2)\left(s^2-\frac37\right)\,\dd s=0,
\end{aligned}
\tag{7.6}
\]
so \(p_0(-1)=1\).

Below the switching line, the layer field is a positive scalar multiple of
\(y(-y,x)\), and
\[
 \frac{\dd}{\dd t}(x^2+y^2)=0.
\tag{7.7}
\]
The upper unit semicircle therefore runs from \((1,0)\) to \((-1,0)\),
closing the singular cycle.  It grazes \(y=1\) at \((0,1)\), where
\[
 \kappa(0)=1,\qquad
 \beta(0)=dh\!\left(\frac{X^+-X^-}{h}\right)=\beta.
\tag{7.8}
\]

The smooth reference return in the negative critical-endpoint coordinate is
\(R_\mu(x)=-p_\mu(x)\).  Differentiating \((7.5)\) gives
\[
 p_\mu'(x)=\frac{f_\mu(x)}{f_\mu(p_\mu(x))}.
\tag{7.9}
\]
Since \(m_0(\pm1)=1\), one has
\[
 p_0'(-1)=-1,\qquad
 p_0''(-1)=f_0'(-1)-f_0'(1)=\frac{16\alpha}{7}.
\tag{7.10}
\]

The fast radius after balance is \(p_\mu(x)\).  Hence the exact moving
penetration is
\[
 q=Q_\mu(x):=p_\mu(x)-1.
\tag{7.11}
\]
Let \(G_\mu(r)>0\) be the absolute value of the negative \(x\)-coordinate
where the physical fast orbit starting at \((r,0)\) returns to \(y=0\).
Proposition~\ref{prop:grazing} gives
\[
 G_\mu(1+q)
 =1+q+\frac{4\sqrt2}{3}\beta\,q_+^{3/2}
 +q_+^2B(\sqrt{q_+},\mu).
\tag{7.12}
\]
The return conjugated by \((7.11)\) is
\[
\begin{aligned}
 \widehat F_\mu(q)
 &=Q_\mu\circ[-G_\mu\circ p_\mu]\circ Q_\mu^{-1}(q)\\
 &=p_\mu\!\left(-G_\mu(1+q)\right)-1.
\end{aligned}
\tag{7.13}
\]
At \(\mu=0\), equations \((7.10)\)--\((7.13)\) yield
\[
 \widehat F_0(q)-q
 =\frac{8\alpha}{7}q^2+
 \frac{4\sqrt2}{3}\beta q_+^{3/2}
 +q_+^2\widetilde B(\sqrt{q_+})+O(q^3),
\tag{7.14}
\]
which proves the first two formulas in \((7.3)\).

To compute the unfolding, put
\[
 I_j=\int_{-1}^1s\psi_j(s)\,\dd s,\qquad
 I_0=\frac23,\quad I_1=0.
\tag{7.15}
\]
Implicit differentiation of \((7.5)\) at \((-1,0)\) gives
\[
 \partial_{\mu_j}p=-I_j,\qquad
 \partial_{\mu_j}\bigl[-p_\mu'(-1)-1\bigr]
 =\psi_j(-1)-\psi_j(1)+f_0'(1)I_j.
\tag{7.16}
\]
For
\[
 a(\mu)=p_\mu(-1)-1,\qquad
 b(\mu)=-p_\mu'(-1)-1,
\]
this becomes
\[
 D_\mu(a,b)(0)=
 \begin{pmatrix}
 -2/3&0\\[1mm]
 -4/3-16\alpha/21&-4/5
 \end{pmatrix},
\tag{7.17}
\]
whose determinant is \(8/15\).

It remains to check the physical entry--exit interface.  In the box
\(y<1\), set \(u=y\), \(v=-x\).  Then
\[
 \dot u=u[-m_\mu(-v)v],\qquad
 \dot v=-\eps+u[m_\mu(-v)u],
\tag{7.18}
\]
which is exactly \((2.1)\) with \(g=-1\) and the required signs.  The fast
fibers have critical endpoints \(1\) and \(-1\).  Choose disjoint
critical-line neighborhoods
\(U_0^{\rm c}\ni1\), \(U_1^{\rm c}\ni-1\), and set
\(V_i=\Lambda_i(U_i^{\rm c})\) after shrinking.  These are disjoint
physical-coordinate neighborhoods on the two components of \(u=u_1\).
The endpoint maps, with their domains displayed, are
\[
\begin{aligned}
 \Lambda_0:U_0^{\rm c}\to V_0,\quad
 \Lambda_0(\zeta)&=+\sqrt{\zeta^2-u_1^2},&
 \pi_0:V_0\to U_0^{\rm c},\quad
 \pi_0(v)&=+\sqrt{u_1^2+v^2},\\
 \Lambda_1:U_1^{\rm c}\to V_1,\quad
 \Lambda_1(\zeta)&=-\sqrt{\zeta^2-u_1^2},&
 \pi_1:V_1\to U_1^{\rm c},\quad
 \pi_1(v)&=-\sqrt{u_1^2+v^2}.
\end{aligned}
\tag{7.19}
\]
Since \(u_1>0\), the maps \(\Lambda_i\) are nonidentity.  They conjugate
the physical-section return to \((7.13)\), so no hidden fiber factor changes
\((7.3)\)--\((7.4)\).  In particular,
\(\mathcal B_\mu:U_0^{\rm c}\to U_1^{\rm c}\) and
\(\mathcal T_{\rm ee}^0:V_0\to V_1\) have exactly the types stated in
Section~2.

Place the return section on the regular negative-\(x\) arc after grazing
and before entry--exit.  The intervening arcs are compact regular tubes.
The reference post-grazing endpoint map in \(I=r-1\) is \(T_0(I)=-1-I\),
while \(Q_0'(-1)=p_0'(-1)=-1\); hence
\[
 L_0=Q_0'(-1)T_0'(0)=1.
\tag{7.20}
\]
For small \(p\), \(\kappa(p)=1-\eps>0\).  Closure, submersion, and the
local-first-return property follow after shrinking the selected tubes and
sections.  Theorem~\ref{thm:entryexit} supplies the required
positive-\(\eps\) passage regularity.  All hypotheses of
Theorem~\ref{thm:main} are therefore satisfied.
\end{proof}

\section{Curvature transfer from the smooth reference family}

The coefficient \(c\) in Theorem~\ref{thm:main} can be imposed directly.
For the biological family it is most naturally checked through the HHY
stability derivative.  The following proposition is the precise
normalization bridge to the smooth-reference integral used by
Huang--Huzak--Yao \cite{HuangHuzakYao2026}; it is not needed for the general
zero count and is not asserted for an arbitrary ambient stability
functional.

Use the HHY upper and lower sections \(\Sigma_1\) and \(\Sigma_2\),
with their source and target coordinates oriented as in
\cite{HuangHuzakYao2026}.  Let
\[
 F^-,B^-:\Sigma_1\longrightarrow\Sigma_2
\]
be, respectively, the long forward entry--exit map and the smooth-reference
backward turn map, both written in the common target coordinate on
\(\Sigma_2\).  If \(s\) is the source coordinate on \(\Sigma_1\),
set
\[
 -F_s^-=e^{\alpha_-},\qquad -B_s^-=e^{\alpha_+}.
\]
Thus \(F_s^-<0\) and \(B_s^-<0\) in these orientations.  Write
\(x_0=x(s)\) for the cycle label and \(q=Q(s)\) for the normalized
penetration coordinate, with \(x(s_0)=1\) and \(Q(s_0)=0\).

\begin{proposition}[Smooth-reference curvature transfer]
\label{prop:lambda}
Suppose the minus branch extends smoothly across the switching line and
produces a smooth family of reference cycles with label \(x_0\) on an open
interval containing the grazing value \(x_0=1\).  Assume the ordinary-cycle
difference-map construction of Section~4.2 of
\cite{HuangHuzakYao2026} applies to this reference family.  At a balanced
neutral reference cycle suppose
\[
 \chi^-(1)=0,\qquad \lambda^-(1)=0,
\tag{8.1}
\]
and let \(s_0\) be its section label.  If
\(x'(s_0)>0\) and \(B_s^-(s_0,0)Q'(s_0)\ne0\),
then
\[
 c
 =-\frac{\Phi_0x'(s_0)}
 {2B_s^-(s_0,0)Q'(s_0)}\,(\lambda^-)'(1),
\qquad \Phi_0>0.
\tag{8.2}
\]
In particular \(c\ne0\) if and only if \((\lambda^-)'(1)\ne0\).
\end{proposition}

\begin{proof}
Put
\[
 \delta_{\rm H}(s,\eps)=F^-(s,\eps)-B^-(s,\eps).
\]
Equations (4.13)--(4.14) of \cite{HuangHuzakYao2026} factor the first
derivative in their crossing construction as
\[
 \partial_s\delta_{\rm H}
 =-\Phi(\alpha_-,\alpha_+)(\alpha_--\alpha_+),
\qquad \Phi>0,
\tag{8.3}
\]
and their equation (4.19) gives
\[
 \partial_s(\alpha_--\alpha_+)(s_0,0)
 =\lambda'(x_0)x'(s_0),\qquad x'(s_0)>0.
\tag{8.4}
\]
These identities occur on printed pp.~13--15.  For ordinary cycles,
Proposition~4.3 and the paragraph following equation (4.23) state that the
same factorization and differentiation argument applies.  Applying that
ordinary-cycle construction to the smooth minus-reference family at its
interior label \(x_0=1\) yields
\[
 \delta_{{\rm H},ss}(s_0,0)
 =-\Phi_0x'(s_0)(\lambda^-)'(1).
\tag{8.5}
\]
This use of the ordinary smooth reference family is distinct from the
physical tangent map in Section~4.3 of that source.

The reference return is
\[
 P^{\rm ref}=(B^-)^{-1}\circ F^-.
\]
Since \(F^-=B^-+\delta_{\rm H}\),
\[
 P^{\rm ref}(s)-s=M(s)\delta_{\rm H}(s,0),
\qquad M(s_0)=\frac1{B_s^-(s_0,0)}.
\tag{8.6}
\]
At the double zero, differentiating twice introduces no derivative of
\(M\).  Conjugating by \(q=Q(s)\) and differentiating at the double fixed
point gives
\[
 S_{qq}(0,0)
 =\frac{\delta_{{\rm H},ss}(s_0,0)}
 {B_s^-(s_0,0)Q'(s_0)}.
\tag{8.7}
\]
Equations \((8.5)\)--\((8.7)\) prove \((8.2)\).
\end{proof}

\begin{corollary}[Oriented sign transfer]
\label{cor:lambda-sign}
Under the hypotheses of Proposition~\ref{prop:lambda}, if
\[
 x'(s_0)>0,\qquad Q'(s_0)>0,\qquad B_s^-(s_0,0)<0,
\]
then
\[
 \sgn c=\sgn (\lambda^-)'(1).
\]
\end{corollary}

\begin{proof}
In \((8.2)\), the factors \(\Phi_0\), \(x'(s_0)\), and \(Q'(s_0)\)
are positive, whereas \(B_s^-(s_0,0)\) is negative.  Hence the scalar
multiplying \((\lambda^-)'(1)\) is positive.
\end{proof}

\section{A cutoff Gause realization}

This section is an application, not an additional premise of the abstract
theorem.  We first identify the published cutoff-response class from which
our base member is selected, then state the new quartic response assumption,
and finally separate
the analytic implicit-function argument, the dual-backend base certificate,
the validated nonzero singular parameter points, and the ordinary
positive-parameter numerical illustration.  No empirical calibration is
claimed.

Throughout this section, \(p(x)\) denotes the biological feeding response;
the abstract return parameter \(p=(\eps,\mu)\) is written explicitly as
\((\eps,\mu)\).  The branch parameter \(\varrho\) is unrelated to the
penetration chart \(\rho\).  First variations of the prey isocline \(g\)
are denoted by \(\delta g\) and \(\delta_j g\), reserving \(q\) for the
moving return coordinate used in the abstract theorem.

We turn to a Gause model whose state variables are prey density \(x\ge0\)
and predator density \(y\ge0\):
\[
 \dot x=x\left(1-\frac{x}{3}\right)-p(x)y,\qquad
 \dot y=y(-\eps+p(x)).
\tag{9.1}
\]
Here \(x\) and \(y\) are scaled prey and predator densities.  The feeding
threshold and cap height are normalized to \(x=1\) and \(p(1)=1\), the prey
growth coefficient and carrying capacity are fixed at \(\phi=1\) and \(k=3\),
and the conversion/yield factor is absorbed into the predator scaling.  The
source mortality parameter \(\delta\) is denoted by \(\eps\).  The biological
slow--fast regime is \(0<\eps\ll1\); throughout the proof, \(\eps=0\) denotes
only the singular analysis limit, not a biological zero-mortality model.
The coefficients \(a_0,a_1,\varrho\) and all response-derivative bounds below
are dimensionless in these normalized coordinates.

Kooij--Zegeling \cite[equation (1.5)]{KooijZegeling2019} study the more
general carrying-capacity version of \((9.1)\) with the continuous cutoff
\[
 p_{\rm KZ}(x)=
 \begin{cases}
 x[1+(x-1)(a_0+a_1x)],&0\le x\le1,\\
 1,&x>1.
 \end{cases}
\tag{9.2}
\]
Their discussion on pp.~164--165 motivates the plateau by
satiation/handling and allows further cutoff responses with the same
qualitative shapes.  The anchor used below is the
\((a_0,a_1)=(-1/2,0)\) member of this normalized family.  It satisfies their
monotonicity condition, but is not one of the three parameter choices
\((0,0),(-1,0),(1,-2)\) analyzed in detail in that article.  We add one
coefficient:
\[
 p(x)=
 \begin{cases}
 x[1+(x-1)(a_0+a_1x+\varrho x^2)],&0\le x\le1,\\
 1,&x\ge1.
 \end{cases}
\tag{9.3}
\]
The term \(\varrho x^2\) inside the brackets, and the parameter selection
below, are new assumptions of this paper.  They are not part of
\cite{KooijZegeling2019}.

Figure~\ref{fig:gause-response} displays the particular base member selected
here from the Kooij--Zegeling response class, embedded as the \(\varrho=0\)
member of \((9.3)\).  Kooij--Zegeling introduce the class \((9.2)\), but do
not single out or analyze this parameter choice.  The theorem branch
\(p_\varrho\), with
\((a_0,a_1)=(A_0(\varrho),A_1(\varrho))\) and
\(0<|\varrho|\ll1\), is a nearby analytic quartic family and is not plotted.

\ifsiadsreview
  \begin{figure}[H]
\else
  \begin{figure}[!b]
\fi
\centering
\begin{minipage}[t]{0.455\textwidth}
\centering
\begin{tikzpicture}[x=3.15cm,y=2.75cm,>=Latex,font=\scriptsize]
 \draw[->] (-0.03,0)--(1.63,0) node[right] {\(x\)};
 \draw[->] (0,-0.03)--(0,1.23) node[above] {\(p(x)\)};
 \draw[very thick,blue!70!black]
   plot[domain=0:1,samples=140] (\x,{0.5*\x*(3-\x)});
 \draw[very thick,blue!70!black] (1,1)--(1.53,1);
 \draw[dashed,gray] (1,0)--(1,1.12);
 \fill[blue!70!black] (1,1) circle (1.25pt);
 \draw (-0.018,1)--(0.018,1) node[left=3pt] {\(1\)};
 \node[below left=1pt] at (0,0) {\(0\)};
 \node[below] at (1,0) {\(x=1\)};
 \node[below left=2pt] at (1,1) {\((1,1)\)};
 \node[align=center,blue!70!black] at (0.47,0.34)
   {selected base member \(p_*\)};
 \node[align=center] at (0.48,0.14)
   {\(p_*'>0,\quad p_*''<0\quad(0<x<1)\)};
 \node[above,blue!70!black] at (1.31,1.03) {\(p_*(x)=1\)};
 \draw[->,gray!80] (0.73,1.16)--(0.96,1.01);
 \node[above left,align=right] at (0.73,1.16)
   {\(p'_{*-}(1)=1/2\)};
 \node[right] at (1.03,0.86) {\(p'_{*+}(1)=0\)};
\end{tikzpicture}

\smallskip
{\small\textbf{(a)} Formula-defined cutoff response.}
\end{minipage}\hfill
\begin{minipage}[t]{0.515\textwidth}
\centering
\begin{tikzpicture}[x=4.7cm,y=1.28cm,>=Latex,font=\scriptsize,
 singular/.style={gray!72,densely dashed,thick},
 fast/.style={violet!85!black,very thick}]
 \draw[->] (-0.04,0)--(1.25,0) node[right] {prey \(x\)};
 \draw[->] (0,-0.03)--(0,2.78) node[above] {predator \(y\)};

 \draw[blue!65!black,very thick] (0,2.63)--(0,0.73);
 \draw[red!70!black,very thick,dashed] (0,0.60)--(0,0.055);
 \draw[blue!65!black,->,thick] (0,1.95)--(0,1.48);
 \draw[red!70!black,->,thick] (0,0.49)--(0,0.25);
 \node[anchor=west,blue!65!black] at (0.045,1.72) {attracting \(S_0^a\)};
 \node[anchor=west,red!70!black] at (0.045,0.36) {repelling \(S_0^r\)};
 \fill (0,0.666667) circle (1.05pt);
 \node[anchor=west] at (0.045,0.73) {exchange \(y=2/3\)};

 \draw[dashed,gray!75] (1,0)--(1,2.66);
 \node[anchor=east,gray!65!black] at (0.98,2.52)
   {seam \(x=1\)};
 \draw[fast,postaction={decorate},decoration={markings,
   mark=at position .27 with {\arrow{>}},
   mark=at position .72 with {\arrow{>}}}]
   plot[domain=-2.41020293:1.34739675,samples=190]
   ({1+(2/3)*(\x-exp(\x)+1)},{(2/3)*exp(\x)});

 \fill[violet!85!black] (0,0.0598647) circle (1.15pt);
 \node[anchor=west,text=violet!85!black] at (0.035,0.18)
   {\(y_\alpha\)};
 \fill[violet!85!black] (0,2.5649312) circle (1.15pt);
 \node[anchor=west,text=violet!85!black] at (0.025,2.64)
   {\(y_\omega\)};
 \fill[violet!85!black] (1,0.666667) circle (1.2pt);
 \draw[->,violet!75!black,thin] (0.78,0.91)--(0.975,0.69);
 \node[anchor=east,text=violet!85!black] at (0.78,0.93)
   {tangency \((1,2/3)\)};

 \node[align=center,text=violet!85!black] at (0.72,2.18)
   {exact fast arc\\(desingularized time)};
\end{tikzpicture}

\smallskip
{\small\textbf{(b)} Formula-defined singular anchor.}
\end{minipage}
\caption{Two formula-generated views of the degenerate base member selected
in this paper from the
Kooij--Zegeling cutoff-response class, at \(\varrho=0\) and
\((a_0,a_1)=(-1/2,0)\).  This parameter choice and the analysis shown here
are not results of \cite{KooijZegeling2019}.  In (a),
\(p_*(x)=x(3-x)/2\) for \(0\le x\le1\) and \(p_*(x)=1\) for \(x\ge1\),
so \(p'_{*-}(1)=1/2\) and \(p'_{*+}(1)=0\).  This derivative mismatch enters
the grazing coefficient.  In (b), the fast arc of the singular cycle is
generated from the exact desingularized solution
\(x_*(t)=1+\frac23(t-e^t+1)\), \(y_*(t)=\frac23e^t\), with
\(e^{T_\pm}-T_\pm=\frac52\) and \(T_-\le t\le T_+\).  It runs from
\((0,y_\alpha)\), through the
quadratic tangency \((1,2/3)\), to \((0,y_\omega)\).  The segment on \(x=0\)
runs downward from \(y_\omega\) to \(y_\alpha\), completing the oriented
singular cycle; its transverse stability changes at \(y=2/3\).  No numerical
ODE integration is used.  The panels show the degenerate base anchor, not a
theorem point; the nearby analytic quartic branch is not shown.}
\label{fig:gause-response}
\end{figure}

Two HHY characteristic functions organize the realization below.  The
quantity \(\chi(x_0)\) is the endpoint-balance mismatch: its vanishing says
that the entry--exit slow connection closes the smooth reference orbit with
cycle label \(x_0\).  The quantity \(\lambda^-(x_0)\) is the HHY
smooth-minus stability integral, computed using the smooth continuation of
the minus field.  Thus
\(\chi(1)=\lambda^-(1)=0\) selects a balanced neutral tangent cycle, while
\((\lambda^-)'(1)\ne0\) supplies its first nonzero smooth curvature.  For
the present Gause family these two functions have the explicit orbit
formulas in \eqref{eq:bio-characteristics} below.

\FloatBarrier

\siadsneedspace{.58\textheight}
\begin{theorem}[Qualitatively admissible Gause cutoff realization]
\label{thm:biology}
There are \(\varrho_0>0\) and real-analytic functions
\[
 A_0,A_1:(-\varrho_0,\varrho_0)\longrightarrow\R,\qquad
 A_0(0)=-\frac12,\quad A_1(0)=0,
\tag{9.4}
\]
such that, for each \(0<|\varrho|<\varrho_0\), the response \((9.3)\)
with \((a_0,a_1)=(A_0(\varrho),A_1(\varrho))\) satisfies
\[
 \chi(1)=0,\qquad
 \lambda^-(1)=0,\qquad
 (\lambda^-)'(1)\ne0.
\tag{9.5}
\]
Here \(\lambda^-\) is the HHY stability integral of the smooth polynomial
continuation of the left branch across \(x=1\).

The response is positive for \(x>0\), strictly increasing and strictly
concave on \(0\le x\le1\), continuous at its cap, and
\[
 p'_-(1)>0=p'_+(1).
\tag{9.6}
\]
The singular fast orbit through \((1,2/3)\) has a nonzero continuous
quadratic-grazing coefficient.  Moreover,
\[
 \det
 \frac{\partial(\chi(1),\lambda^-(1))}
 {\partial(a_0,a_1)}
 \ne0
\tag{9.7}
\]
along the selected branch, and this rank transfers to
\(\det\partial_\mu(a,b)\ne0\) for the canonical return coefficients.
Consequently, after translating
\(\mu=(a_0,a_1)-(A_0(\varrho),A_1(\varrho))\), system \((9.1)\)--\((9.3)\)
satisfies the model-specific hypotheses of Theorem~\ref{thm:main}.
In the normalized return coordinate,
\[
 d>0,\qquad
 \sgn c=\sgn (\lambda^-)'(1)=-\sgn\varrho,
 \qquad
 \sgn(cd)=-\sgn\varrho.
\tag{9.7a}
\]
In addition to the unspecified local branch radius, the two concrete choices
\(\varrho=\pm1/20\) admit unique balanced-neutral parameter pairs in the
explicit boxes of Proposition~\ref{prop:bio-validated-centers}.  Each
certified parameter point satisfies \((9.5)\)--\((9.7\mathrm a)\) pointwise and,
after translating \(\mu\) about that pair, satisfies the model-specific
hypotheses of Theorem~\ref{thm:main}.  This pointwise certification does not
assert a validated continuation from the local analytic branch to
\(\varrho=\pm1/20\).
\end{theorem}

\begin{proposition}[Analytic balance--neutrality branch]
\label{prop:bio-analytic-branch}
Consider the base response
\[
 (a_0,a_1,\varrho)=(-1/2,0,0).
\]
Near this point, the two equations
\(\chi(1)=\lambda^-(1)=0\) determine a real-analytic branch
\((a_0,a_1)=(A_0(\varrho),A_1(\varrho))\).  The characteristic-function
Jacobian
\[
 D_{(a_0,a_1)}\bigl(\chi(1),\lambda^-(1)\bigr)
\]
is nonsingular at the base point and, after shrinking, along this branch.
Along the branch,
\((\lambda^-)'(1)\ne0\) for every sufficiently small \(\varrho\ne0\), and
\[
 \sgn (\lambda^-)'(1)=-\sgn\varrho.
\]
\end{proposition}

\begin{proof}

Put
\[
 B(x)=1+(x-1)(a_0+a_1x+\varrho x^2),\qquad
 g(x)=\frac{1-x/3}{B(x)}.
\tag{9.8}
\]
The left polynomial is used on a fixed open neighborhood of \([0,1]\)
whenever a derivative in the cycle label is taken.  At \(\eps=0\) and
for \(x>0\), division of the field by \(p_L(x)>0\) gives
\[
 \dot x=g(x)-y,\qquad \dot y=y.
\tag{9.9}
\]
Since \(p_L(x)=xB(x)\) with \(B(0)>0\), the right-hand side of
\((9.9)\) extends smoothly to \(x=0\).  Below, \((9.9)\) denotes this
smooth desingularized extension on a neighborhood containing the endpoint
line \(x=0\); no literal division by \(p_L(0)\) is asserted.
The orbit tangent at \(x_0=1\) has \(y(0)=g(1)=2/3\).  Let
\(T_-<0<T_+\) be its two hits on \(x=0\).  The HHY characteristic
functions reduce to
\[
\begin{aligned}
 \chi(1)
 &=\frac23(e^{T_+}-e^{T_-})-g(0)(T_+-T_-),\\
 \lambda^-(1)
 &=\int_{T_-}^{T_+}g'(X(t))\,\dd t.
\end{aligned}
\tag{9.10}
\label{eq:bio-characteristics}
\]

At
\[
 (a_0,a_1,\varrho)=\left(-\frac12,0,0\right),
\tag{9.11}
\]
one has
\[
 p_L(x)=\frac{x(3-x)}2,\qquad g(x)\equiv m:=\frac23.
\tag{9.12}
\]
The tangent orbit and its hitting times are
\[
 X_*(t)=1+m(t-e^t+1),\qquad
 e^{T_\pm}-T_\pm=\frac52.
\tag{9.13}
\]
Thus \(\chi(1)=\lambda^-(1)=0\), but
\((\lambda^-)'(1)=0\): this selected base anchor from the published
cutoff-response class satisfies balance and neutrality but fails the
simple-curvature condition required here.

For \(j=0,1,2\), define the first parameter variations
\[
 \delta_j g(x)=\frac43\,\frac{(1-x)x^j}{3-x}.
\tag{9.14}
\]
If a parameter direction changes \(g\) by \(\delta g\), differentiation of the
orbit and transverse endpoint equations gives
\[
\begin{aligned}
 D\chi[\delta g]
 &=\int_{T_-}^{T_+}\bigl(\delta g(X_*(t))-\delta g(0)\bigr)\,\dd t,\\
 D\lambda^-[\delta g]
 &=\int_{T_-}^{T_+}(\delta g)'(X_*(t))\,\dd t.
\end{aligned}
\tag{9.15}
\]
Endpoint variations cancel in the first equation, while the base identity
\(g'_*=0\) removes the orbit variation from the second.

Let \(J\) have columns
\[
 J_j=
 \begin{pmatrix}
 D\chi[\delta_j g]\\ D\lambda^-[\delta_j g]
 \end{pmatrix},
\qquad j=0,1.
\]
The interval certificate described in Appendix~\ref{app:certificate}
proves
\[
 \det J\in[1.1380086664,\,1.2575279529].
\tag{9.16}
\]
This certifies only the base-point derivative, not a numerical \(\varrho_0\).
The vector field, transverse hitting times, and integrals are
real-analytic in \((a_0,a_1,\varrho)\).  The analytic implicit-function
theorem therefore produces \((9.4)\), for some unspecified sufficiently
small \(\varrho_0>0\), with the first two equalities in \((9.5)\).
Nonrigorous central values, included only for orientation, are
\[
 A_0'(0)\approx-0.0394525821,\qquad
 A_1'(0)\approx-0.7187633790.
\tag{9.17}
\]

To differentiate the stability integral in the cycle label, define
\[
\begin{aligned}
 \mathfrak M(\delta g)
 &=
 \int_{T_-}^{T_+}(\delta g)''(X_*(t))\,\dd t\\
 &\quad+(\delta g)'(0)\left[
 -\frac1{m(1-e^{T_+})}
 +\frac1{m(1-e^{T_-})}
 \right].
\end{aligned}
\tag{9.18}
\]
This follows from
\[
 \frac{\dd T_\pm}{\dd x_0}
 =-\frac1{m(1-e^{T_\pm})}.
\tag{9.19}
\]
Along the implicit branch, the first isocline variation is
\[
 \delta_{\rm br} g=\delta_2 g+A_0'(0)\delta_0 g+A_1'(0)\delta_1 g.
\tag{9.20}
\]
The same interval certificate proves
\[
 \mathfrak M(\delta_{\rm br} g)
 \in[-2.6409284065,\,-1.5038804530].
\tag{9.21}
\]
This is again a base-point certificate; analyticity and the expansion below
propagate its sign for unspecified small nonzero \(\varrho\), without a
certified numerical radius.
Consequently
\[
 (\lambda^-)'(1;\varrho)
 =\varrho\,\mathfrak M(\delta_{\rm br} g)+O(\varrho^2),
\tag{9.22}
\]
which is nonzero for all sufficiently small \(\varrho\ne0\), with the sign
claimed in the proposition.
\end{proof}

\begin{proposition}[Validated nonzero singular balanced-neutral parameter points]
\label{prop:bio-validated-centers}
For the exact values \(\varrho_+=1/20\) and \(\varrho_-=-1/20\), put
\[
\begin{aligned}
 \mathcal B_+={}&
 [-0.50197275145282727,-0.50197275125282725]\\
 &\quad{}\times[-0.035947758939822244,-0.035947758739822241],\\[2pt]
 \mathcal B_-={}&
 [-0.49802749511679928,-0.49802749491679926]\\
 &\quad{}\times[0.035928499838247668,0.035928500038247671].
\end{aligned}
\tag{9.22a}
\]
For each sign, there is exactly one pair \((a_0,a_1)\in\mathcal B_\pm\)
such that \(\chi(1)=\lambda^-(1)=0\).  At these two parameter points, validated
enclosures give
\[
\begin{array}{c|c|c}
 \varrho&\det D_{(a_0,a_1)}(\chi,\lambda^-)&(\lambda^-)'(1)\\ \hline
 1/20&[1.1941593312,1.1941593352]
 &[-0.1048330993,-0.1048330959]\\
 -1/20&[1.2001493803,1.2001493843]
 &[0.1050308568,0.1050308603]
\end{array}
\tag{9.22b}
\]
The lower/upper axis hits are transverse, \(B(x)\ge1\), and the remaining
shape and seam enclosures are
\[
\begin{gathered}
\begin{array}{c|c|c}
 \varrho&\min_{[0,1]}p'&\max_{[0,1]}p''\\ \hline
 1/20&0.5120062563&-0.8476526041\\
 -1/20&0.4878590474&-0.9570619837
\end{array}\\[4pt]
\begin{array}{c|c|c}
 \varrho&p'_-(1)&\beta_{\rm bio}\\ \hline
 1/20&[0.5120794896,0.5120794901]&[0.3413863264,0.3413863267]\\
 -1/20&[0.4879010047,0.4879010052]&[0.3252673364,0.3252673368]
\end{array}
\end{gathered}
\tag{9.22c}
\]
Here the displayed minimum is a validated lower bound and the displayed
maximum is a validated upper bound, so both points satisfy the biological
shape, grazing, curvature-sign, and
rank hypotheses of Theorem~\ref{thm:biology}.
\end{proposition}

\begin{proof}
Let \(F_\varrho(a_0,a_1)=(\chi(1),\lambda^-(1))\).  The certificate
augments the desingularized layer equations by
\(z'=g'(x)\), \(a_0'=a_1'=0\), and computes validated
\(C^1\) Poincar\'e maps from \((1,2/3,0,a_0,a_1)\) to \(x=0\)
for the forward and time-reversed fields using
CAPD::DynSys~\cite{KapelaMrozekWilczakZgliczynski2021}.  Their values and
derivatives give interval extensions of \(F_\varrho\) and
\(DF_\varrho\), including the implicit hitting-time correction.  Chaining
the same derivative enclosures with the initial tangent
\((1,g'(1),0,0,0)\) gives the intervals for \((\lambda^-)'(1)\)
in \((9.22\mathrm b)\).

Regard every endpoint in \((9.22\mathrm a)\) as the exact decimal rational
displayed there.  Let \(\operatorname{fl}_{64}\) denote IEEE~754 binary64
round-to-nearest, ties-to-even, and let \(\operatorname{nextDown}\), respectively
\(\operatorname{nextUp}\), denote the adjacent binary64 number
toward \(-\infty\), respectively \(+\infty\).  The validated computation is
performed on the exact dyadic superbox \(\widehat{\mathcal B}_\pm\) obtained
by replacing each lower endpoint \(l\) by
\(\operatorname{nextDown}(\operatorname{fl}_{64}(l))\) and each upper
endpoint \(u\) by
\(\operatorname{nextUp}(\operatorname{fl}_{64}(u))\).  An exact-rational
semantic check verifies
\(\mathcal B_\pm\subset\operatorname{int}\widehat{\mathcal B}_\pm\).

For the recorded dyadic centers \(m_\pm\in\mathcal B_\pm\), the interval
Newton images satisfy the stronger inclusions
\[
 m_\pm-[DF_{\varrho_\pm}(\widehat{\mathcal B}_\pm)]^{-1}
 F_{\varrho_\pm}(m_\pm)
 \subset\operatorname{int}\mathcal B_\pm
 \subset\mathcal B_\pm
 \subset\operatorname{int}\widehat{\mathcal B}_\pm
\tag{9.22d}
\]
hold strictly; all interval pivots exclude zero.  The interval Newton
theorem~\cite[Chapter~8]{MooreKearfottCloud2009} therefore gives one zero in
each superbox; the first inclusion places it in \(\mathcal B_\pm\), while
uniqueness on the superbox proves uniqueness in \(\mathcal B_\pm\).
The determinant and curvature intervals are
then immediate from the same validated \(C^1\) maps.  Finally, natural
interval extensions on 4096 exact dyadic cells covering \([0,1]\) give
the shape bounds in \((9.22\mathrm b)\)--\((9.22\mathrm c)\); the
validated endpoint velocities satisfy \(\dot x>0.6061\) at the lower hit
and \(\dot x<-1.8979\) at the upper hit.  The source, complete enclosures,
and a pinned reproduction script are supplied with the paper.  The
machine-readable certificate records the exact decimal theorem boxes as
strings and the dyadic CAPD superboxes separately.
\end{proof}

\begin{proposition}[Admissibility, grazing, and coefficient signs]
\label{prop:bio-geometry-sign}
After shrinking the branch of Proposition~\ref{prop:bio-analytic-branch},
each response with \(\varrho\ne0\) is positive, strictly increasing, and
strictly concave before the cap, with \(p'_-(1)>0=p'_+(1)\).
The same statements hold pointwise at the two certified parameter points of
Proposition~\ref{prop:bio-validated-centers}.  The tangent
orbit has a uniformly nondegenerate quadratic contact on every fixed range
\(0\le\eps\le\eps_0<1\).  In the normalized return coordinate its smooth
curvature and grazing coefficients satisfy
\[
 c=\frac{\Phi_0}{2e^{\alpha_+}}(\lambda^-)'(1),
 \qquad d=\frac{4\sqrt3}{3}\,\beta_{\rm bio}>0,
\]
and consequently
\[
 \sgn c=-\sgn\varrho,
 \qquad \sgn(cd)=-\sgn\varrho.
\]
\end{proposition}

\begin{proof}

At the base point,
\[
 B_*(x)=\frac{3-x}{2}\ge1,\qquad
 p'_{L,*}(x)=\frac{3-2x}{2}\ge\frac12,\qquad
 p''_{L,*}(x)=-1
\tag{9.23}
\]
on \([0,1]\).  These strict margins persist along the implicit branch.
At the two certified parameter points, the corresponding strict margins are instead
given directly by \((9.22\mathrm b)\)--\((9.22\mathrm c)\).
Hence \(p(0)=0\), \(p(x)>0\) for \(x>0\), and the response is strictly
increasing and concave before saturation.  Both coordinate axes are
invariant, and the vector field points into the strip \(0\le x\le3\) along
\(x=3\).

At the tangent point \(y_*=g(1)=2/3\),
\[
\begin{aligned}
 p'_-(1)&=1+a_0+a_1+\varrho,\qquad p'_+(1)=0,\\
 \beta_{\rm bio}
 &=\frac23(1+a_0+a_1+\varrho),\\
 \kappa_\eps&=\frac23(1-\eps).
\end{aligned}
\tag{9.24}
\]
Thus \(\beta_{\rm bio}\) is close to \(1/3\), while
\(\kappa_\eps>0\) for \(0\le\eps<1\).  For every fixed
\(0<\eps_0<1\),
\[
 \kappa_\eps\ge\frac23(1-\eps_0)>0
 \qquad(0\le\eps\le\eps_0),
\]
so the quadratic contact is uniformly nondegenerate on precisely these
restricted intervals.  Proposition~\ref{prop:grazing} gives a nonzero
\(3/2\) coefficient.

Choose the HHY upper section \(\Sigma_1\) as the return section \(\Pi\),
and use the normalized first integral in \((2.10)\) with
\(h=x-1\).  At a selected balanced-neutral parameter point,
\[
 Q'(s_0)=x'(s_0)>0,\qquad L_0=1.
\]
Consequently, in the notation of Proposition~\ref{prop:lambda},
\[
 c=\frac{\Phi_0}{2e^{\alpha_+}}(\lambda^-)'(1),
 \qquad
 d=\frac{4\sqrt3}{3}\,\beta_{\rm bio},
\]
and hence \(\sgn(cd)=\sgn (\lambda^-)'(1)\).

It remains to verify the displayed transfer formulas.
Let \(\Sigma_2\) be the lower HHY section and write
\(F^-,B^-:\Sigma_1\to\Sigma_2\) for the long forward map through the
entry--exit block and the smooth-reference backward turn map, respectively.
Let \(A:\Sigma_2\to I_{\rm in}\) and
\(C:I_{\rm out}\to\Sigma_1\) be the adjacent regular maps.  The reference
grazing passage is the identity in the \(I\)-coordinate because
\(X^-I=0\).  Therefore
\[
 Q=A\circ F^-,\qquad
 C\circ A=(B^-)^{-1},\qquad
 A\circ B^-=J\circ x,
 \qquad J(x_0):=I(x_0,g(x_0)).
\]
At the grazing point \(z_*=(1,g(1))\), the normalization
\(dI(z_*)=dh(z_*)=dx\) gives
\[
 J'(1)=dI(z_*)\,(1,g'(1))=1.
\]
HHY use the orientation
\(-F_s^-=e^{\alpha_-}\), \(-B_s^-=e^{\alpha_+}\).  Neutrality gives
\(\alpha_- -\alpha_+=\lambda^-(1)=0\), and hence
\(F_s^-(s_0)=B_s^-(s_0)=-e^{\alpha_+}\).  Differentiating the preceding
composition identities at the base orbit yields
\[
 Q'(s_0)=A'F_s^-=A'B_s^-=J'(1)x'(s_0)=x'(s_0)>0
\]
and
\[
 L_0=(Q\circ C)'=Q'C'
 =\frac{Q'}{A'B_s^-}=1.
\]
Substitution into \((8.2)\) gives the displayed formula for \(c\).
Equation \((2.16)\), together with \(\kappa_0=2/3\), gives the formula
for \(d\).  Finally \(\beta_{\rm bio}>0\) by \((9.24)\), proving the sign
identity.
\end{proof}

\begin{proposition}[Biological-to-canonical rank transfer]
\label{prop:bio-rank-transfer}
Let \(\bar\mu_\varrho\) be either a sufficiently small nonzero
balanced-neutral parameter point
\((A_0(\varrho),A_1(\varrho))\) on the analytic branch or the unique
parameter point in \(\mathcal B_\pm\) at \(\varrho=\pm1/20\).  At
\(\bar\mu_\varrho\), the differential of the
canonical coefficients factors as
\[
 D_\mu(a,b)
 =
 \begin{pmatrix}
 C_0&0\\ C_\chi&C_1
 \end{pmatrix}
 D_\mu(\chi,\lambda^-),
 \qquad C_0C_1\ne0.
\]
In particular, the biological rank-two condition \((9.7)\) transfers to
the canonical unfolding condition \((2.14)\).
\end{proposition}

\begin{proof}
The point is that rank in the two characteristic functions is the rank
required for the canonical coefficients, rather than merely a convenient
surrogate.  Take differentials at
\((\eps,\mu)=(0,\bar\mu_\varrho)\).  Let \(y_\omega(\mu)\) and
\(y_\alpha(\mu)\) be the upper and lower \(x=0\) hits of the smooth
minus-reference orbit.  Define
\[
 H_\mu(y)=y-g(0;\mu)\log y,
\qquad
 H_\mu(E_\mu(y;\mu))=H_\mu(y),
\quad E_\mu(y;\mu)<g(0;\mu)<y.
\tag{9.25}
\]
At balance,
\(E_{\bar\mu_\varrho}(y_\omega(\bar\mu_\varrho);\bar\mu_\varrho)
=y_\alpha(\bar\mu_\varrho)\).  Substitution of \(y=y_\omega(\mu)\) in
\((9.25)\) and subtraction of \(H_\mu(y_\alpha(\mu))\) give the identity
\[
 \chi(1;\mu)
 =H_\mu\!\left(E_\mu(y_\omega(\mu);\mu)\right)
  -H_\mu(y_\alpha(\mu)).
\]
Differentiate this identity at balance.  The two arguments then coincide,
so the identical pure-parameter terms
\(\partial_\mu H_\mu(y_\alpha)\) cancel; the remaining chain-rule terms give
\[
 \dd\!\left(E_\mu(y_\omega(\mu);\mu)-y_\alpha(\mu)\right)
 =\frac1{H'_{\bar\mu_\varrho}(y_\alpha(\bar\mu_\varrho))}\,\dd\chi.
\tag{9.26}
\]
The denominator is nonzero because
\(y_\alpha(\bar\mu_\varrho)<g(0;\bar\mu_\varrho)\).

Transport this endpoint mismatch through the regular flow maps and the
penetration coordinate.  There is a smooth
\(\mathcal A_{\rm out}(h,\mu)\) with
\[
 a(\mu)=\mathcal A_{\rm out}(h(\mu),\mu),\qquad
 \mathcal A_{\rm out}(0,\mu)=0,\qquad
 \partial_h\mathcal A_{\rm out}(0,\bar\mu_\varrho)=\sigma\ne0.
\]
Therefore
\[
 \dd a=C_0\,\dd\chi,\qquad
 C_0=\frac{\sigma}
 {H'_{\bar\mu_\varrho}(y_\alpha(\bar\mu_\varrho))}\ne0.
\tag{9.27}
\]

On the balanced hypersurface \(\chi=0\), let \(F^-,B^-\) be the HHY
forward and backward section maps and \(s_0(\mu)\) the orbit with label
\(x_0=1\).  The factorization \((8.3)\) and conjugacy invariance of the
fixed-point multiplier give
\[
 b(\mu)
 =-\frac{\Phi(\alpha_-,\alpha_+)}
 {B_s^-(s_0(\mu),\mu)}\,\lambda^-(1;\mu)
 =:r(\mu)\lambda^-(1;\mu),
\qquad r(\bar\mu_\varrho)\ne0.
\tag{9.28}
\]
At the balanced point \(\lambda^-=0\).  Differentiating along vectors tangent
to \(\{\chi=0\}\) shows that
\(\dd b-r(\bar\mu_\varrho)\dd\lambda^-\) annihilates
\(\ker\dd\chi\), hence is a multiple of \(\dd\chi\).  Thus
\[
 D_\mu(a,b)(\bar\mu_\varrho)
 =
 \begin{pmatrix}
 C_0&0\\ C_\chi&C_1
 \end{pmatrix}
 D_\mu(\chi,\lambda^-)(\bar\mu_\varrho),
\qquad C_1=r(\bar\mu_\varrho)\ne0.
\tag{9.29}
\]
The triangular factor is invertible.  Equation \((9.16)\), or
\((9.22\mathrm b)\) at either explicit parameter point, together with
\((9.29)\), proves canonical rank two.
\end{proof}

\begin{proposition}[Recognition of the abstract singularity]
\label{prop:bio-recognition}
Along the branch of Proposition~\ref{prop:bio-analytic-branch}, and
pointwise at either parameter point of Proposition~\ref{prop:bio-validated-centers},
the shifted coordinates below put the Gause family into the entry--exit
form \((2.1)\) with the signs in \((2.2)\).
After shrinking the local boxes and parameter neighborhoods, the transverse
hits, smooth minus-reference family, regular flow tubes, singular itinerary,
and local return closure satisfy the structural model hypotheses used in
Theorem~\ref{thm:main}.
\end{proposition}

\begin{proof}
Put
\[
 y_c(\mu)=\frac1{p'_L(0;\mu)}
\]
and use the shifted entry--exit coordinates \(u=x\),
\(v=y-y_c(\mu)\).  System \((9.1)\) is then locally of the exact form
\[
 \dot u=u\left[
 1-\frac u3-\frac{p(u)}u\bigl(v+y_c(\mu)\bigr)\right],
\qquad
 \dot v=-\eps\bigl(v+y_c(\mu)\bigr)
 +u\left[\bigl(v+y_c(\mu)\bigr)\frac{p(u)}u\right].
\tag{9.30}
\]
At \(u=\eps=0\), the slow coefficient
\(-[v+y_c(\mu)]\) is uniformly negative, while the normal coefficient is
\(-p'_L(0;\mu)v\), so the signs in \((2.2)\) hold after shrinking the
shifted \(v\)-interval.
Take the HHY upper section \(\Sigma_1\) as \(\Pi\), and choose disjoint closed
entry--exit and grazing boxes.  The left polynomial supplies the smooth
minus-reference family.  The complementary arcs from the exit face of the
entry--exit box to the grazing entrance and from the grazing exit back to
\(\Sigma_1\) are compact; the explicit base field is nonzero there and all
boundary hits are transverse.
At either explicit parameter point, the validated lower hit has \(y>0.0596\), and
\(\dot y=y\) on the desingularized layer arc; hence that arc is likewise
regular away from the selected grazing box.  Its two axis hits have the
strict transverse velocity bounds stated after \((9.22\mathrm d)\).
Compactness gives uniform field and
transversality bounds
and bounded flight times.  Shrinking the boxes and parameter neighborhood
then excludes equilibria and unintended \(\Sigma_1\)-hits and, by flow-box
uniqueness, preserves \eqref{eq:itinerary}.  Thus every nearby selected orbit
meets \(\Pi=\Sigma_1\) exactly once per circuit, verifying \textup{(REG)};
the selected base orbit supplies local return closure.
\end{proof}

\begin{proof}[Proof of Theorem~\ref{thm:biology}]
Proposition~\ref{prop:bio-analytic-branch} supplies the analytic
balanced-neutral branch, the nonsingular characteristic-function Jacobian,
and the simple smooth curvature with its certified sign.  By continuity the
Jacobian remains nonsingular after reducing \(\varrho_0\).
Proposition~\ref{prop:bio-validated-centers} independently supplies the
two explicit nonzero singular parameter points, without asserting a validated continuation
from \(\varrho=0\) to either endpoint.
Proposition~\ref{prop:bio-geometry-sign} supplies biological admissibility,
uniform quadratic grazing, and
\(d>0\), \(\sgn c=-\sgn\varrho\).  Proposition~\ref{prop:bio-rank-transfer}
transports the certified rank to the canonical coefficients, while
Proposition~\ref{prop:bio-recognition} verifies the remaining local
entry--exit, reference-family, itinerary, and closure hypotheses.  Together
these modules give \((9.4)\)--\((9.7a)\) and every model-specific
hypothesis of Theorem~\ref{thm:main}.
\end{proof}

\begin{corollary}[Biological two-versus-three realization]
\label{cor:biology-cycles}
Fix a sufficiently small nonzero \(\varrho\) on the analytic branch, or
take either certified parameter point at \(\varrho=\pm1/20\).  In the selected local return
neighborhood, the cyclicity bound is three if \(\varrho>0\) and two if
\(\varrho<0\).  For every sufficiently small fixed \(\eps>0\), response
parameters arbitrarily close to the corresponding distinguished
positive-\(\eps\) parameter point of the rank-two unfolding realize,
respectively, three or two simple limit cycles in that neighborhood.  The
realizing responses may be chosen
positive, strictly increasing and strictly concave before the cap, with
\(p'_-(1)>0=p'_+(1)\).
\end{corollary}

\begin{proof}
By Theorem~\ref{thm:biology},
\(\sgn(cd)=-\sgn\varrho\).  The bounds follow from
Theorem~\ref{thm:main}; its rank-two sharpness conclusion gives the asserted
simple physical cycles for every sufficiently small fixed \(\eps>0\).
These strict admissibility conditions are open near the balanced response,
so the rank-two attaining paths can be chosen inside that neighborhood.
\end{proof}

\begin{remark}[Biological coordinates]
The two unfolding parameters can be replaced by endpoint response slopes:
\[
 p'(0)=1-a_0,\qquad
 p'_-(1)=1+a_0+a_1+\varrho.
\tag{9.31}
\]
The affine change has determinant \(-1\).  Thus the rank-two unfolding is
equivalently expressed by the low-density attack slope and the
pre-saturation marginal feeding slope.
\end{remark}

\subsection{Positive-parameter numerical illustration}
\label{subsec:bio-positive-numerics}

We finally illustrate, rather than prove, the physical unfolding at
\(\eps=0.01\).  The computations use the section \(x=0.05\),
\(\dot x>0\), logarithmic state variables, and explicit event handling for
the itinerary
\[
 p_L\longrightarrow x=1\longrightarrow p_R=1
 \longrightarrow x=1\longrightarrow p_L.
\tag{9.32}
\]
If \(\zeta_{\rm gr}\) denotes the section value of \(\log y\) whose smooth
left-reference orbit grazes \(x=1\), the plotted moving coordinate and
physical displacement are
\[
 q=\zeta_{\rm gr}-\log y,\qquad
 \Delta_{\rm num}(q)=\log y-P(\log y).
\]
Thus \(q>0\) is the penetrating side and \(q<0\) is the nonpenetrating
side.  This numerical coordinate is not used to certify the analytic return
germ.
The two parameter choices are
\[
\begin{aligned}
 \varrho=1/20:&\quad
 (a_0,a_1)=(-0.5014790869078154,-0.0343916980511351),\\
 \varrho=-1/20:&\quad
 (a_0,a_1)=(-0.4980524993175209,0.0360172862534296).
\end{aligned}
\]
These are nearby unfolding parameters, not the singular balanced-neutral points in
\(\mathcal B_\pm\).  Direct minimization of the two sub-cap response
polynomials gives
\[
\begin{array}{c|cc}
\varrho&\min_{[0,1]}p'(x)&\max_{[0,1]}p''(x)\\ \hline
 1/20&0.514129215041\ldots&-0.840524966020\ldots\\
-1/20&0.487964786936\ldots&-0.957154968126\ldots
\end{array}
\tag{9.32a}
\]
so both actual plotted responses are increasing and strictly concave below
the cap.  This algebraic input-shape check does not validate the computed
cycles.  The computed fixed points are
\[
\begin{array}{c|c|c|c|c|c}
\varrho&q&\max x&\text{period}&P'&\text{itinerary}\\ \hline
 1/20&-6.3251925\times10^{-2}&0.9618695&369.6452&1.0018580&\text{nonpenetrating}\\
 1/20& 1.7335246\times10^{-6}&1.0000010&377.9672&0.9992197&\text{penetrating}\\
 1/20& 7.9367233\times10^{-6}&1.0000048&377.9680&1.0006119&\text{penetrating}\\
-1/20&-5.6216264\times10^{-3}&0.9965971&376.6046&0.9996455&\text{nonpenetrating}\\
-1/20& 1.4199781\times10^{-4}&1.0000863&377.3617&1.0105841&\text{penetrating}
\end{array}
\tag{9.33}
\]
where \(P'\) is a finite-difference Poincar\'e multiplier.  Thus the
computed stability order is unstable--stable--unstable in the three-cycle
case and stable--unstable in the two-cycle case.

Every stored zero is enclosed by an opposite-sign displacement bracket.
The brackets persist when the relative ODE tolerance is multiplied by
\(1\), \(0.3\), and \(0.1\); recomputation on \(x=0.04\) and
\(x=0.06\) gives a maximum return residual below \(4.5\times10^{-11}\),
and the maximum multiplier spread over three difference steps is below
\(10^{-6}\).  These are ordinary floating-point robustness checks.  They
do not replace the analytic local-cyclicity theorem, certify the displayed
positive-\(\eps\) fixed points, or give a global cycle count.

\ifsiadsreview
  \begin{figure}[p]
\else
  \begin{figure}[!htbp]
\fi
\centering
\includegraphics[width=\textwidth]{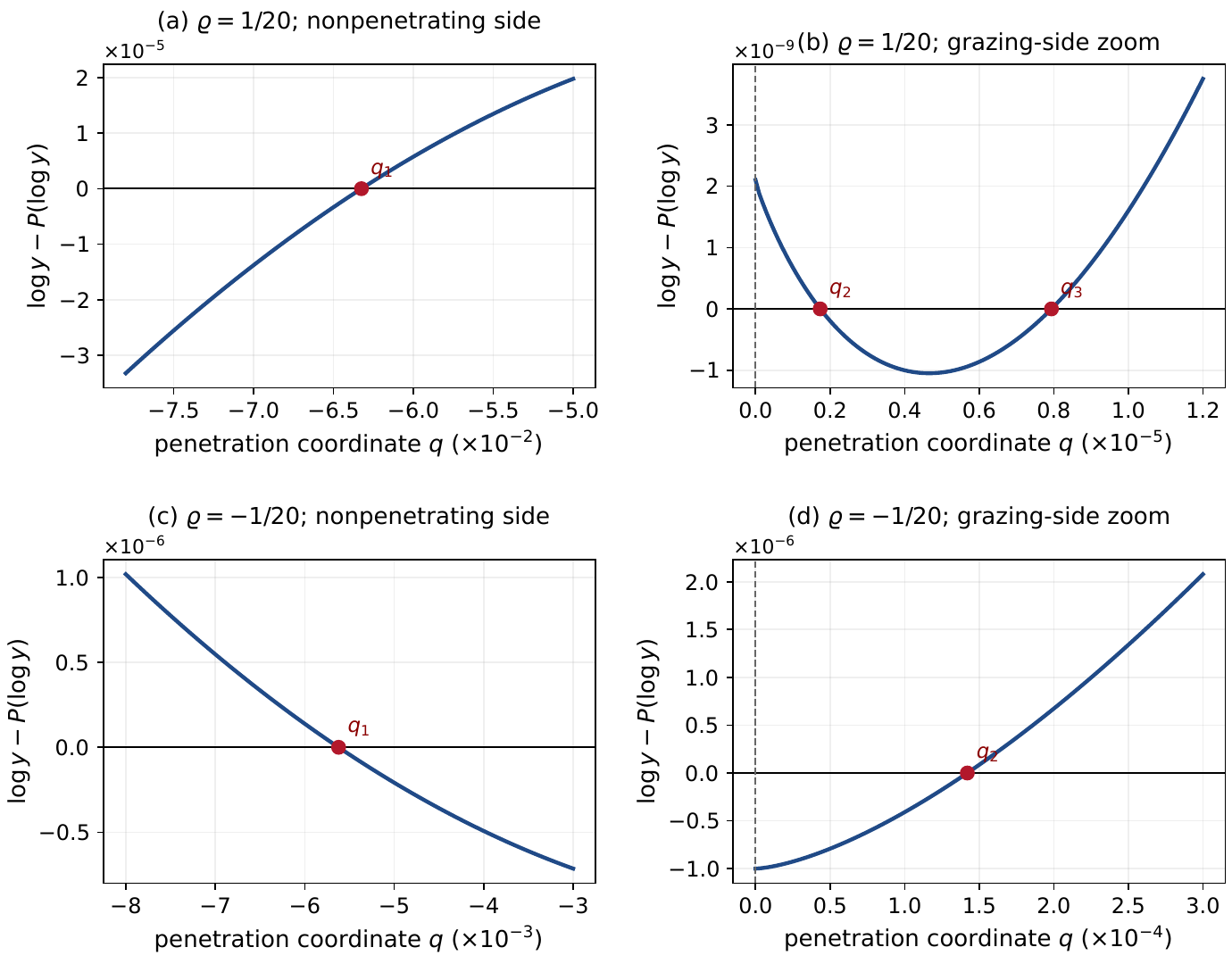}
\caption{Physical cutoff-Gause Poincar\'e displacements at \(\eps=0.01\)
from ordinary event-driven numerical integration.  Here
\(q=\zeta_{\rm gr}-\log y\) and the ordinate is
\(\Delta_{\rm num}(q)=\log y-P(\log y)\), so \(q<0\) is nonpenetrating and
\(q>0\) is penetrating.  Panels~(a)--(b) use
\(\varrho=1/20\): panel~(a) shows the nonpenetrating side and panel~(b)
shows the grazing-side zoom, together containing one computed
nonpenetrating zero and two computed penetrating zeros.  Their parameters are
\((a_0,a_1)=(-0.5014790869078154,\allowbreak
-0.0343916980511351)\).  Panels~(c)--(d) give the corresponding
nonpenetrating and grazing-side views for \(\varrho=-1/20\),
\((a_0,a_1)=(-0.4980524993175209,\allowbreak
0.0360172862534296)\), with one computed zero of each type.  Both responses
satisfy the shape inequalities in \((9.32\mathrm a)\).  In panels~(b) and
(d), the dashed line is the moving grazing threshold \(q=0\); red filled
circles in all panels mark the fixed points listed in \((9.33)\), with the
\(q_j\) labels ordered separately within each value of \(\varrho\).  Blue
solid curves join sampled displacement values.  Separate scales are necessary because
the penetrating zeros lie much closer to the threshold.  This vector figure
is an ordinary numerical illustration, not a validated orbit count or an
interval certificate.}
\label{fig:bio-positive-displacements}
\end{figure}
\FloatBarrier

\begin{remark}[Scope of the application]
The quartic response is mathematically and qualitatively biologically
admissible, but it is not empirically calibrated here.  The theorem supports
the Gause equations, small-death interpretation, saturation mechanism, and
finite-cyclicity application; it does not claim that the codimension-two
parameter choice has already been observed in a specific ecosystem.  The
exact two- or three-cycle realizations in
Corollary~\ref{cor:biology-cycles} occur at nearby unfolding parameters, not
at the singular balanced-neutral point itself, and are local rather than global cycle
counts.
\end{remark}

\section{Discussion and open problems}

The theorem identifies a reusable composite singularity.  The entry--exit
module supplies a parameter-uniform physical passage, the grazing module
supplies a ramified \(3/2\) correction, and the exact moving penetration
coordinate permits their derivative-controlled composition.  At a neutral
return, the invariant sign of \(cd\) then distinguishes two sharp local
cyclicity classes.  The polynomial benchmark shows that neither class is
empty, while the cutoff Gause family demonstrates how the intrinsic
coefficients and the rank-two unfolding can be transferred from a concrete
model.

Several extensions suggest themselves.
\begin{enumerate}
\item The entry--exit theorem is proved for the exact invariant-line form
\((2.1)\).  A coordinate-invariant recognition theorem for more general
critical curves would make the result easier to apply beyond this normal
form.
\item The entry--exit and grazing boxes are separated.  If another singular
passage lies on the grazing-to-return leg, a new transmission argument is
needed to determine the analogue of \(L_0\) and the resulting return germ.
\item Higher-order contact, higher-order branch matching, multiple switching
curves, and smooth regularization should lead to other ramified return
classes.  Determining their sharp cyclicity is a natural next problem.
\item For the cutoff Gause family, empirical calibration and a global
bifurcation atlas would complement the local realization proved here.
\end{enumerate}

Proposition~\ref{prop:lambda} uses the HHY minus-extension normalization;
equivalence with a differently normalized ambient stability functional is not
asserted.  The finite-cyclicity conclusion is local in phase space and uniform
in a small parameter neighborhood, rather than a bound for all cycles of the
global vector field.

\appendix

\section{Interval certificate for the cutoff branch}
\label{app:certificate}

This appendix records the reproducible inequalities used in
\((9.16)\) and \((9.21)\).  It is a computer-assisted interval
calculation, not a proof-assistant formalization.

\ifsiadsreview
Two independent interval implementations certify the same exact rational
brackets for the roots of \(e^t-t=5/2\),
\[
\begin{aligned}
 T_-&\in[-2.41020292977619,\,-2.41020292977618],\\
 T_+&\in[ 1.34739675103134,\, 1.34739675103135].
\end{aligned}
\tag{A.1}
\]
The first uses \texttt{mpmath.iv}; the second uses
\texttt{python-flint/Arb}, exact rational cell endpoints, and direct
interval enclosures of the endpoint tails.  Their independently computed
enclosures have the proof-critical intersections
\[
\begin{aligned}
 C_0L_1-C_1L_0
 &\in[1.138008666466789,\,1.25752795285509]>0,\\
 \mathfrak M(\delta_{\rm br}g)
 &\in[-2.640928406431336,\,-1.503880453020846]<0.
\end{aligned}
\tag{A.2}
\]
Each backend separately proves both strict signs.  The accompanying code
supplement contains the complete \(C_j,L_j,R_j\) tables, exact rational
outer bounds, recorded output, platform and package versions, file hashes,
the cross-backend intersection gate, and exact reproduction commands.
\else
The two roots of \(e^t-t=5/2\) are enclosed by
\[
\begin{aligned}
 T_-&\in[-2.41020292977619,\,-2.41020292977618],\\
 T_+&\in[ 1.34739675103134,\, 1.34739675103135].
\end{aligned}
\tag{A.1}
\]
Monotonicity of \(e^t-t\) on the two relevant half-lines, together with
opposite endpoint signs, certifies uniqueness in these intervals.

For \(j=0,1,2\), the calculation encloses
\[
\begin{aligned}
 C_j&=\int_{T_-}^{T_+}
 \bigl(\delta_j g(X_*(t))-\delta_j g(0)\bigr)\,\dd t,\\
 L_j&=\int_{T_-}^{T_+}(\delta_j g)'(X_*(t))\,\dd t,\\
 R_j&=\mathfrak M(\delta_j g).
\end{aligned}
\tag{A.2}
\]
It uses natural interval extensions with 30 decimal digits, a
500-cell composite interval Riemann sum on the central interval
\[
 [-2.41020292977618,\,1.34739675103134],
\]
and the conservative tail enclosure
\([-10^{-10},10^{-10}]\).  The omitted endpoint lengths are below
\(2\cdot10^{-14}\), and the script checks that every relevant endpoint
integrand has absolute interval upper bound below \(1000\).

The resulting intervals are
\[
\begin{array}{c|c|c|c}
j&C_j&L_j&R_j\\ \hline
0&
[-0.996428,-0.972075]&
[-1.902796,-1.885933]&
[-2.338513,-2.315763]\\
1&
[0.271001,0.293857]&
[-0.700883,-0.644927]&
[-4.115831,-4.047581]\\
2&
[0.155048,0.173401]&
[-0.600101,-0.516527]&
[-5.232193,-5.016704]
\end{array}
\tag{A.3}
\]
and interval Gaussian elimination gives
\[
\begin{aligned}
 C_0L_1-C_1L_0
 &\in[1.1380086664667,1.2575279528551],\\
 A_0'(0)&\in[-0.067090,-0.014669],\\
 A_1'(0)&\in[-0.815373,-0.631807],\\
 \mathfrak M(\delta_{\rm br} g)
 &\in[-2.6409284064314,-1.5038804530208].
\end{aligned}
\tag{A.4}
\]
Only the strict positivity of the first interval and strict negativity of
the last interval enter the proof.

The repository checkout reproduces both independent certificates by
\begin{verbatim}
python3 experiments/certify_cutoff_extension.py
python3 experiments/certify_cutoff_extension_arb.py
python3 experiments/crosscheck_cutoff_certificates.py
\end{verbatim}
In the arXiv source package, the ancillary copy is run by
\begin{verbatim}
python3 anc/certify_cutoff_extension.py
python3 anc/certify_cutoff_extension_arb.py
python3 anc/crosscheck_cutoff_certificates.py
\end{verbatim}
The reviewed local run used Python~3.11.14, \texttt{mpmath}~1.3.0, and
\texttt{python-flint}~0.9.0.  The two scripts use independent interval
implementations; the cross-check requires identical exact root brackets,
the expected strict sign from each backend, and nonempty intersections for
every reported coefficient interval.
\fi

\section{Regularity bookkeeping}

For clarity, the derivative losses used in the proof are:
\[
\begin{array}{ccl}
C^{r+4}\text{ entry--exit data}
&\longrightarrow&
C^{r+4}\text{ layer-footpoint coordinate}\\
&\longrightarrow&
C^{r+3}\text{ standard slow--fast field}\\
&\longrightarrow&
C^{r+1}\text{ prepared endpoint charts}\\
&\longrightarrow&
C^r\text{ physical passage}.
\end{array}
\tag{B.1}
\]
The post-grazing composition in \((5.2)\) consumes one derivative, so a
\(C^\ell\) ramified coefficient uses \(r=\ell+1\).  The curvature
argument takes \(\ell=2\).  Neighborhoods may depend on the prescribed
finite order; no diagonal argument asserting one common \(C^\infty\)
neighborhood is used.

\section*{Declarations}

\paragraph{Funding.}
The author received no funding for this work.

\paragraph{Competing interests.}
The author declares no competing interests.

\paragraph{Generative AI use.}
During the development and preparation of this work, the author used
\mbox{OpenAI Codex} to assist with literature discovery, proof stress-testing,
symbolic and numerical code review, manuscript organization, and language
editing.  The author reviewed all AI-assisted outputs and independently
checked the sources, proofs, computations, and code used in the article.
The author assumes responsibility for all content.

\paragraph{Data and code availability.}
The interval-certification sources, recorded enclosures, and reproduction
instructions for the model-specific Gause inequalities and the two explicit
singular balanced-neutral parameter points
\ifsiadsreview
accompany this manuscript as supplementary material.
\else
are included as ancillary files with this preprint.
\fi
The positive-\(\eps\) displacement data and the scripts that generate
Figure~\ref{fig:bio-positive-displacements}
\ifsiadsreview
are included in the supplement as a separately labelled component
\else
are included in the ancillary files as a separately labelled component
\fi
and are identified as ordinary floating-point evidence.  The general cyclicity theorem
is analytic and does not depend on numerical orbit integration.

\begingroup
\ifsiadsreview
  \let\originalthebibliography\thebibliography
  \renewcommand{\thebibliography}[1]{%
    \originalthebibliography{#1}%
    \setlength{\itemsep}{0pt}%
    \setlength{\parskip}{0pt}%
  }
  \footnotesize
\else
  \small
\fi
\bibliographystyle{siam}
\bibliography{references}
\endgroup

\end{document}